\documentclass[10pt]{article}
\usepackage{amsmath,amsthm,amssymb}
\usepackage[T1]{fontenc}
\usepackage[utf8]{inputenc}
\usepackage[dvipdf]{graphicx}
\usepackage{color}
\usepackage{epstopdf}

\usepackage{enumerate}
\usepackage{enumitem}

\usepackage[normalem]{ulem}

\definecolor{db}{RGB}{0, 0, 130}
\usepackage[colorlinks=true,citecolor=red,linkcolor=db,urlcolor=blue,pdfstartview=FitH]{hyperref}

\usepackage[numbers]{natbib}

\definecolor{rp}{rgb}{0.25, 0, 0.75}
\definecolor{dg}{rgb}{0, 0.6, 0}

\newcommand{\R}{\mathbb{R}}

\newcommand{\N}{\mathbb{N}}
\newcommand{\EE}{\mathbb{E}}

\newcommand{\Md}{\mathbb{M}_d}

\newcommand{\dn}{\delta_n}
\newcommand{\ens}{\eta_n(s)}
\newcommand{\ent}{\eta_n(t)}
\newcommand{\enu}{\eta_n(u)}
\newcommand{\enr}{\eta_n(r)}

\makeatletter
\newcommand{\customlabel}[2]{%
   \protected@write \@auxout {}{\string\newlabel {#1}{{#2}{\thepage}{#2}{#1}{}}}%
   \hypertarget{#1}{#2\hspace{-0.13cm}}
}
\makeatother

\newtheorem{theorem}{Theorem}[section]

\newtheorem{definition}{Definition}[section]

\newtheorem{corollary}[definition]{Corollary}
\newtheorem{example}[definition]{Example}
\newtheorem{assumption}[theorem]{Assumption}
\newtheorem{lemma}[definition]{Lemma}

\newtheorem{proposition}[definition]{Proposition}
\newtheorem{remark}[definition]{Remark}

\DeclareMathOperator{\Tr}{Trace}
\def\E{\mathbb{E}}
\def\x{\times}
\def\Om{\Omega}
\def\Fc{\mathcal{F}}
\def\F{\mathbb{F}}
\def\P{\mathbb{P}}
\def\Xb{\overline{X}}
\def\eps{\varepsilon}
\def\Xh{\widehat{X}}
\def\Ph{\widehat{P}}
\def\Yh{\widehat{Y}}

\author{Alexandre Richard\footnote{Universit\'e Paris-Saclay, CentraleSup\'elec, MICS and CNRS FR-3487, France. \texttt{alexandre.richard@centralesupelec.fr}.}
\and Xiaolu Tan\footnote{Department of Mathematics, The Chinese University of Hong Kong. \texttt{xiaolu.tan@cuhk.edu.hk}, research supported by CUHK startup grant and CUHK Faculty of Science Direct Grant 2020-2021.}
\and Fan Yang\footnote{Department of Mathematics, The Chinese University of Hong Kong. \texttt{fyang@math.cuhk.edu.hk}.}
}

\title{ \Large{\textbf{Discrete-time Simulation of Stochastic Volterra Equations}}\footnote{We are grateful to Eduardo Abi Jaber for helpful discussions, and to two anonymous reviewers for their useful comments and suggestions.}}

\begin{document}

\maketitle

\begin{abstract}

	We study discrete-time simulation schemes for stochastic Volterra equations,
	namely the Euler and Milstein schemes, and the corresponding Multilevel Monte-Carlo method.
	By using and adapting some results from Zhang \cite{Zhang}, together with the Garsia-Rodemich-Rumsey lemma,
	we obtain the convergence rates of the Euler scheme and Milstein scheme under the supremum norm.
	We then apply these schemes to approximate the expectation of functionals of such Volterra equations by the (Multilevel) Monte-Carlo method, and compute their complexity.
	 We finally  provide some numerical simulation results.

\end{abstract}

\noindent \textbf{Key words:} Stochastic  Volterra equations, Euler scheme, Milstein scheme,  Monte-Carlo method, MLMC.

\medskip

\noindent\textbf{MSC2010 subject classification:} 60H20 ; 65C05 ; 65C30.

\section{Introduction}

	We study the discrete-time approximation problem for  stochastic Volterra equations of the form
	\begin{align}\label{eq:defX_intro}
		X_t = X_0 + \int_0^t K_1(t,s) b(s,X_s)~ds + \int_0^t K_2(t,s) \sigma(s,X_s)~dW_s,~t\in[0,T],
	\end{align}
	by means of the Euler scheme, the Milstein scheme and the corresponding Multilevel Monte-Carlo method.
	In the above equation, $X$ is an $\R^d$-valued process, $W$ is a $d$-dimensional standard Brownian motion,
	 $K_1, K_2$ are (possibly singular) kernels,
	 and $b, \sigma$ are coefficient functions whose properties will be detailed below.
	
	\vspace{0.5em}
	
	As natural extension of (deterministic) Volterra equations, the stochastic Volterra equation is motivated by the physics of heat transfer 
	(see for instance the introductory example of the book of Gripenberg, Londen and Staffans \cite{GripenbergEtAl}  with a random source term), 
	the physics of dissipative dynamics and anomalous diffusions (see for instance Jak{\v{s}}i{\'c} and Pillet \cite{JaksicPillet}, resp. Lutz \cite{Lutz}),
	and has been studied since the works of Berger and Mizel \cite{BergerMizel} and Protter \cite{Protter} in the non-singular kernels and Lipschitz coefficients case.
	Let us also mention the recent rough volatility modelling in mathematical finance,
	which leads to some 	affine Volterra equations, see e.g. El Euch and Rosenbaum \cite{ElEuchRosenbaum}, and Abi Jaber, Larsson, and Pulido \cite{jaber2019affine}.
	
	\vspace{0.5em}

	The main objective of the paper is to study the discrete-time simulation problem for the stochastic Volterra equation \eqref{eq:defX_intro}.
	Observe that when $K_1 \equiv K_2 \equiv I_d$, the Volterra equation degenerates into a standard SDE,
	and the corresponding Monte-Carlo simulation problem has been tremendously studied during the last decades.
	In general, the simulation of SDEs is based on discrete-time schemes,
	and to estimate the expectation of a functional of an SDE by Monte-Carlo method,
	one has two kinds of error: the discretization error and the statistical error.
	The statistical error is proportional to $\frac{1}{\sqrt{N}}$, where  $N$ is the number of simulated copies of the SDE, by an application of the Central Limit Theorem.
	The discretization error depends essentially on the time step $\Delta t$.
	For the most simple Euler scheme, a (weak) convergence rate of the discretization error has been initially obtained by Talay and Tubaro \cite{talay1990expansion}.
 	Since then, many works have been devoted to study various schemes under different conditions.
	For an overview on this subject, let us refer to Kloeden and Platen \cite{kloeden2013numerical}, Graham and Talay \cite{graham2013stochastic}, and also Jourdain and Kohatsu-Higa \cite{JourdainKH} for a recent review.
	To reduce the discretization error, one needs to use finer discretization,
	which increases the computational complexity for the simulation of the process, and hence increases the statistical error given a fixed total computation effort.
	Then one needs to make a trade-off between the two errors to minimize the total error.

	\vspace{0.5em}

	To improve the usual trade-off between the two errors, Giles \cite{giles2008multilevel} introduced the so-called Multilevel Monte-Carlo (MLMC) method,
	which has been applied and improved in various situations, and has generated a stream of literature,
	see e.g. Giles and Szpruch \cite{GilesSzpruch}, Alaya and Kebaier \cite{alaya2015central}, etc.
	The main idea of the MLMC method is to consider different levels of the time discretization, and rewrite the finest discrete scheme as a telescopic sum of differences between consecutive levels, and then to choose the number of simulations at each level in an optimal way. Let us mention that MLMC has already been studied in the setting of SDEs driven by fractional Brownian motions (denoted later by fBm): first in Kloeden, Neuenkirch and Pavani \cite{KloedenNeuenkirchPavani} with Hurst exponent $H>\tfrac{1}{2}$ and additive fractional noise, and then extensions to rough SDEs in Bayer, Friz, Riedel, and Schoenmakers \cite{BayerFrizEtAl}. 
	This latter article corresponds to a Hurst exponent $H\in(\tfrac{1}{4},\tfrac{1}{2})$, which is still far from the observed roughness of the volatility ($H\approx 0.1$, see Gatheral, Jaisson and Rosenbaum \cite{GatheralJaissonRosenbaum}). The advantage of the Volterra approach compared to integration w.r.t. fBm is that one can achieve very low path regularities, while an equivalent approach through rough paths would be restricted, in practice (although not theoretically), to $H>\tfrac{1}{4}$ (\cite{BayerFrizEtAl}).

	\vspace{0.5em}

	In this paper, we will study the discretization error of the Euler scheme and the Milstein scheme for the stochastic Volterra equation \eqref{eq:defX_intro} with any H\"older regularity (Hurst exponent) between $0$ and $1$,
	and then adapt the MLMC technique in our context.
	For the stochastic Volterra equation in a more general form, the corresponding Euler scheme has already been  studied by Zhang \cite{Zhang},
	where the main results state that, for the uniform discretization scheme with time step $\Delta t = 2^{-n} T$, 
	the discretization error is bounded by $C 2^{- n\eta}$, for some constant $\eta > 0$ (which is not given explicitly but might be found in the proof).
	In this paper, we let $(X^n_t)_{0 \le t \le T}$ denote the solution of the Euler scheme with a general (not necessary uniform) discretization $\pi_n$,
	and adapt the techniques in \cite{Zhang} to our context to obtain an explicit convergence rate of $\E \big[ | X_t - X^n_t |^p \big]$ for each fixed $t \in [0,T]$ and $p \ge 1$.
	Then, in place of the argument with Kolmogorov’s continuity criterium used in \cite{Zhang}, 
	we apply the technique based on the Garsia-Rodemich-Rumsey lemma to obtain an explicit rate for the supremum norm error $\E \big[ \sup_{0 \le t \le T} |X_t - X^n_t|^p \big]$.
	Our new technique provides a better convergence rate than the one in \cite{Zhang}, in particular when not all moments of the initial condition are integrable,
	and the discretization $\pi_n$ could be arbitrary rather than the special uniform discretization of size $T 2^{-n}$ that is required in the technical proof of \cite{Zhang} (see also Remark \ref{rem:EulerScheme}).
	Next, we introduce and extend our techniques and results to a higher order scheme, the Milstein scheme, in order to improve the convergence rate.
	We then study the MLMC method based on the Euler scheme, and compare their computational cost for a given theoretical error.
	These different methods are also tested with various numerical examples.
	We would like also to mention the recent paper \cite{LHH} which appeared at the same time as ours, where the authors study both Euler and Milstein scheme of Volterra equation \eqref{eq:defX_intro} with special kernel $K_1(t,s) := (t-s)^{-\alpha}$ and $K_2(t,s) := (t-s)^{-\beta}$.
	The paper establishes a convergence rate result on $\E \big[ | X_t - X^n_t |^p \big]$ for every fixed $t \in [0,T]$, which is essentially the same as ours.

	\vspace{0.5em}
	
	The rest of the paper is organized as follows. In Section \ref{sec:results}, we state some conditions on $K_{1},~K_{2},~b$ and $\sigma$ that we require for the Euler and Milstein schemes. We then present these two schemes and the corresponding convergence rate results in Theorems \ref{thm:EulerScheme} and \ref{th:convMilstScheme}. 
	In the third part of this section, we detail the Multilevel Monte-Carlo method to approximate quantities of the form $\EE[f(X_{\cdot})]$ and provide some complexity analysis for a given error. 
	Then, in Section \ref{sec:examples}, we provide some numerical examples for these  simulation methods.
	Finally, Section \ref{sec:proofs} gathers the proofs of Theorems \ref{thm:EulerScheme} and \ref{th:convMilstScheme}.

\section{Time discretization of the stochastic Volterra equation and the error analysis}\label{sec:results}

	Let us denote by $\Md$ the set of all $d \x d$-dimensional matrices,
	equipped with the norm $\| \cdot\|$ defined by $\|M\|^2:=\Tr(MM^{\top})$ for all $M \in \Md$.
	The space $\R^d$ is equipped with  the Euclidean norm, denoted by $|\cdot|$ or $\| \cdot \|$ according to the context. Let $T>0$.
	We consider the following  stochastic Volterra equation,
	with the kernels $K_1, K_2:[0,T]^2 \rightarrow \Md$, and  coefficient functions  $b:[0,T]\times \R^d \rightarrow \R^d$, $\sigma:[0,T] \times \R^d \rightarrow \Md$,
	\begin{equation}\label{eq:defX}
		X_t=X_0+\int_0^t K_1(t,s)b(s,X_s)~ds+\int_0^t K_2 (t,s) \sigma(s, X_s)~ dW_s, ~t\in [0,T],
	\end{equation}
	where $W$ is a $d$-dimensional standard Brownian motion in a filtered probability space $(\Om, \Fc, \F, \P)$,
	and the solution $X = (X_t)_{0 \le t \le T}$ is an $\R^d$-valued continuous adapted process.
	Throughout the paper, we assume the conditions on $K_1, K_2, b$ and $\sigma$ in Assumption \ref{assum:main}.
	In particular, under Assumption \ref{assum:main}, the Volterra equation \eqref{eq:defX} has a strong solution, which is unique in $L^p([0,T] \x \Omega)$ for some $p$ large enough 
	(see e.g. Coutin and Decreusefond  \cite[Theorem 3.2]{CoutinDecreusefond} or Wang \cite[Theorem 1.1]{Wang}).
	
	\vspace{0.5em}

	Let us consider, for each $n \ge 1$, a discrete grid $\pi_n = \{0 =t^n_0 < t^n_1 < \dots < t^n_n =T\}$, and denote
	\begin{equation} \label{eq:def_eta_n}
		\delta_n := \max_{0 \le k \le n-1} (t^n_{k+1} - t^n_k),
		~~\mbox{and}~
		\ens:= t^n_k,
		~~\mbox{for}~
		s \in \big[ t^n_k, t^n_{k+1} \big), ~k \ge 0.
	\end{equation}	

	\begin{assumption} \label{assum:main}
		Let $\alpha_1 >0 $, $\alpha_2 >0 $, $\beta_1 > 1$, $\beta_2 > 1$, $\alpha := \alpha_1 \wedge \alpha_2$ and $C > 0$ be fixed constants,
		
		\begin{enumerate}[label=$\mathbf{(A\arabic*)}$]
		\item\label{eq:A1}
		$K_i(t,s) = 0$ whenever $s \ge t$,  and
		\begin{equation*}
			\int_0^t 
			\Big( \|K_1(t,s)\|^{\beta_1} ~+~ \|K_2(t,s)\|^{2 \beta_2} \Big) 
			ds
			<
			\infty,
			~\mbox{for all}~ t \in [0,T];
		\end{equation*}
		
		\item\label{eq:A2}
		for all $t \le t'$, and $n \ge 1$, it holds that
		$$
			\int_t^{t'} 
			\Big(
				\|K_1(t',s)\|
				+
				\|K_1 (t',\ens)\|
			\Big) 
			ds
			~\leq~
			C(t'-t)^{\alpha_1 \wedge 1},
		$$
		and
		$$
			\int_t^{t'}
			\Big(
				\|K_2(t',s)\|^2
				+
				\|K_2 (t',\ens)\|^2
			\Big)
			ds
			\leq
			C(t'-t)^{ 2( \alpha_2 \wedge 1)};
		$$
		
		\item\label{eq:A3} 
		for all $t \in [0,T]$, $n \ge 1$ and $\delta\in(0, \tfrac{t}{2}\wedge(T-t))$, it holds that
		\begin{equation*}
			\int_0^{t} 
			\Big(
				\|K_1(t+\delta,s)-K_1(t,s)\|
				+ 
				\|K_1(t+\delta,\ens)-K_1(t,\ens)\|
			\Big) ds
			~\leq~
			C \delta^{\alpha_1 \wedge 1},
		\end{equation*}
		and
		\begin{equation*}
			\int_0^{t} 
			\Big(
				\|K_2(t+\delta,s)-K_2(t,s)\|^2
				+ 
				\|K_2(t+\delta,\ens)-K_2(t,\ens)\|^2
			\Big) ds
			~\leq~
			C \delta^{2(\alpha_2 \wedge 1)};
		\end{equation*}

		\item\label{eq:A4} for all $t \in [0,T]$, $n \ge 1$, it holds that
		\begin{equation*}
			\int_0^t 
				\| K_1 (t,s) - K_1 (t,\ens) \|
			~ds
			\le C \dn^{ \alpha_1 \wedge 1},
			~~
			\int_0^t
				\| K_2(t,s) - K_2(t,\ens) \|^2
			~ds
			\leq C\dn^{2(\alpha_2 \wedge 1)};
		\end{equation*}
		\end{enumerate}
		
	\begin{enumerate}[label=$\mathbf{(B)}$]
		\item\label{eq:B} for all $s, t \in [0,T]$ and $x, y \in \R^d$, it holds that
		one has  $\|(b, \sigma)(0,0) \| \le C$,
		\begin{equation*}
			\| (b, \sigma) (t,x)- (b, \sigma)(t,y) \| \leq C |x-y|,
			~~\mbox{and}~~
			\|(b, \sigma)(t,x) - (b, \sigma)(s,x) \|
			\leq
			C|t-s|^{\alpha \wedge 1}(1+|x|).
		\end{equation*}
	\end{enumerate}

	\end{assumption}

	\begin{example} \label{exam:Kernel}
		Let 
			$K_1(t , s) = \big((t-s)^{H^1_{j,k}- 1 }\big)_{1 \le j, k \le d} \in \Md$,
			$K_2(t , s) = \big((t-s)^{H^2_{j,k}-\frac12 }\big)_{1 \le j, k \le d} \in \Md$,
			for some positive constants $\{ H^i_{j,k} ~i=1,2, ~1 \le j, k \le d \}$ taking value in $(0, + \infty)$.
		Then 
		it is easy to check that Conditions \ref{eq:A1}-\ref{eq:A4}  hold true with 
		$\alpha_i = \min( H^i_{j,k} ~: 1 \le j, k \le d)$, $i=1,2$,
		$\beta_1 \in (1, \frac{1}{1- \alpha_1 \wedge 1})$,
		$\beta_2 \in (1, \frac{1}{1-2 \alpha_2 \wedge 1})$,
		where $1/0 = \infty$  by  convention.
	\end{example}

\subsection{The Euler scheme}

	As for standard SDEs, the Euler scheme can be obtained by freezing the time between two time points $t^n_k$ and $t^n_{k+1}$  in Equation \eqref{eq:defX}.
	More precisely, for each $n \ge 1$, with $\ens$ defined in \eqref{eq:def_eta_n},
	the solution $X^n$ of the Euler scheme of \eqref{eq:defX} is given by
	\begin{equation} \label{eq:EulerScheme}
		X_t^n=X_0+\int_0^t K_1(t,\ens)~ b(\ens,X_{\ens}^n)~ds+\int_0^t K_2 (t,\ens) ~\sigma(\ens, X_{\ens}^n)~dW_s .
	\end{equation}

	\begin{remark} \label{rem:Euler_Implementation}
		In practice, we will only simulate the value of $X^n$ on the discrete-time grid $\pi_n = \{ t_k, ~k=0, 1, \dots, n\}$,
		and this can be achieved by simulations of the increment of the Brownian motion $\Delta W_{k+1} := W_{t_{k+1}} - W_{t_k}$, $k = 0, \dots, n-1$:
		let $\Delta t_{k+1} := t_{k+1} - t_k$,  $X^n_{t_0} := X_0$, and then
		$$
			X^n_{t_{k+1}}
			~=~
			X_{0}
			~+~
			\sum_{i=0}^{k} K_1\big(t_{k+1}, t_i \big) b(t_i,X^n_{t_i}) \Delta t_{i+1}
			~+~
			\sum_{i=1}^{k} K_2 \big(t_{k+1}, t_i \big) \sigma(t_i, X^n_{t_i} ) \Delta W_{i+1}.
		$$
	\end{remark}

	\begin{theorem}\label{thm:EulerScheme}
		Let Assumption \ref{assum:main} hold true with the constant $\alpha > 0$ defined therein.

		\vspace{0.5em}

		\noindent $(i)$  Let $p \geq  \max( \frac{\beta_1}{\beta_1 -1}, \frac{2\beta_2}{\beta_2 -1})$.
		There exists a constant $C_p \in (0, \infty)$ depending only on $T$, $d$, $p$, and $\beta_1$, $\beta_2$, $C$ in Assumption  \ref{assum:main}
		such that, for all $s, t \in [0,T]$ and $n \ge 1$,
		$$
			\EE \Big[ \big|X^n_t-X^n_s \big|^p \Big] 
			\leq C_p \big(1+\EE \big[|X_0|^p \big] \big)  | t - s |^{p (\alpha\wedge 1)}
			~~\mbox{and}~~
			\EE \Big[ \big| X^n_t-X_t \big|^p \Big]
			\leq
			C_p \big(1+\EE \big[|X_0|^p \big] \big) \dn^{p(\alpha\wedge 1)}.
		$$

		\noindent $(ii)$ Let, in addition, $p> (\alpha\wedge 1)^{-1}$.
		Then for all $\varepsilon\in(\frac{1}{p}, \alpha \wedge 1)$, there exists $C_{p, \varepsilon} \in (0, \infty)$ such that
		\begin{equation} \label{eq:CVG_rate_Euler}
			\Big( \EE \Big[ \sup_{t\in[0,T]}  |X^n_t-X_t|^p \Big] \Big)^{\frac 1p}
			~\leq~
			C_{p, \varepsilon} ~ \big(1+\EE \big[|X_0|^p \big] \big)^{\frac{1}{p}}
			~\dn^{(\alpha\wedge 1)-\varepsilon},
			~~\mbox{for all}~
			n \ge 1.
		\end{equation}
	\end{theorem}

	\begin{remark} \label{rem:EulerScheme}
		$(i)$ When $K_1 = K_2 $ and they are equal to the identity matrix $I_d$,
		so that the Volterra equation \eqref{eq:defX} degenerates into a standard SDE and Assumption \ref{assum:main} holds with $\alpha = \frac{1}{2}$,
		the convergence rate result in Theorem \ref{thm:EulerScheme}.(i) is consistent with results on the strong error of Euler scheme for standard SDEs.

		\vspace{0.5em}
		
		\noindent $(ii)$ The convergence rate result in Theorem \ref{thm:EulerScheme}.(ii) is less general than that of the standard SDEs.
		The main reason is that the solution $X$ of \eqref{eq:defX} is not a semi-martingale in general,
		and the Burkholder-Davis-Gundy inequality fails in this context.
		We instead use the Garsia-Rodemich-Rumsey lemma to obtain an estimation of the strong error on the uniform convergence norm,
		and need to sacrifice a small $\varepsilon > 0$ in the convergence rate.	

		\vspace{0.5em}

		\noindent $(iii)$ A convergence rate result similar to \eqref{eq:CVG_rate_Euler} has also been given in Zhang \cite{Zhang}, but without an explicit expression of the rate.
		Their main idea is to consider a nested sequence of uniform discretizations $\{\Delta t_n := Tk / 2^n ~: k = 0, \dots, 2^n \}$ of the interval $[0,T]$,
		and then consider $X^n_t$ as a random variable indexed by $(X_0, t, \Delta_n)$.
		Using the (multi-dimensional) Kolmogorov’s continuity Theorem, they obtained a strong convergence rate $\Delta t_n^{(\alpha \wedge 1) - \eps}$ under the uniform convergence norm for some un undetermined $\eps$.
		After a careful examination of their proof, their $\eps$ needs to be taken in $( \frac{d+2}{p}, \alpha \wedge 1)$ instead of $(\frac1p, \alpha \wedge 1)$ in our results.
		In this sense, our convergence rate in \eqref{eq:CVG_rate_Euler} is better than that in \cite{Zhang}.
		
		\vspace{0.5em}
		
		Nevertheless, in the case that $\E[ |X_0|^p ]$ is finite for all $p \ge 1$, one can take the constant $p$ large enough, so that $\eps > 0$ can be arbitrarily small for both convergence rate results.
		But still, the technique of \cite{Zhang} requires a special nested uniform discretizations $\{Tk / 2^n ~: k = 0, \dots, 2^n \}$, while we can consider an arbitrary discrete time grid on $[0,T]$.
	\end{remark}

\subsection{The Milstein scheme}

	To obtain a higher order of convergence rate,
	we introduce a Milstein scheme.
	Let us first assume some additional conditions on the coefficient functions.
	
	\begin{assumption} \label{assum:main2}

		Let $\alpha_1 > 0$, $\alpha_2 > 0$, $\alpha > 0$, $C>0$ be the same constants as in Assumption \ref{assum:main}.
		Assume in addition that $\alpha_1 > \frac12$ so that $\alpha' :=  \min\big( \alpha_1, 2\alpha_2,  \alpha_1 + \alpha_2 -\frac12\big) /2 > 0$
		satisfies $\alpha' \le \alpha \le 2 \alpha'$.

		\begin{enumerate}[label=$\mathbf{(A5)}$]
		\item\label{eq:A6}
		For all $0 \le r \le r' \le t \le T$, it holds that 
		\begin{equation*} \label{eq:A6}
			\int_r^{r'} \|K_1(t,s)K_2(s,r)\|~ds
			~\leq~
			C (r'-r)^{2 \alpha' \wedge 1}.
		\end{equation*}
		\end{enumerate}
		
		\begin{enumerate}[label=$\mathbf{\widetilde{(B)}}$]
		\item \label{eq:Bt} The coefficient functions $b$ and $\sigma$ are in $ \mathcal{C}^{0,2}([0,T]\times \R^{d})$,
		and moreover, for all $s,t\in [0,T]$ and $x\in \R^{d}$, it holds that
		$$
			\big\| \nabla_x b(t,x) \big\| + \big\| \nabla^2_{xx} b(t,x) \big\| + \big\| \nabla_x \sigma(t,x) \big\| + \big\| \nabla^2_{xx} \sigma(t,x) \big\| 
			\le C,
		$$
		and
		\begin{equation*}
			\big| (b, \sigma)(t,x)- (b, \sigma)(s,x) \big|
			~\leq~
			 C |t-s|^{2\alpha' \wedge 1}(1+|x|).
		\end{equation*}
		
		\end{enumerate}

	\end{assumption}

	\begin{remark}
		In the context of Example \ref{exam:Kernel},
		when $\alpha_1 :=  \min( H^1_{j,k} ~: 1 \le j, k \le d) > \frac12$ and $\alpha_2 :=  \min( H^2_{j,k} ~: 1 \le j, k \le d) > 0$,
		Condition \ref{eq:A6} still holds true.
		Indeed, in this example, it follows by the Cauchy-Schwarz inequality that
		\begin{eqnarray*}
			\int_r^{r'} \| K_1(t,s) K_2(s, r) \| ds 
			&\le&
			C \int_r^{r'} (t-s)^{\alpha_1 - 1}(s-r)^{\alpha_2 -\frac12} ds \\
			&\le&
			C \sqrt{ \int_r^{r'} (t -s)^{2 \alpha_1 - 2} ds} \, \sqrt{\int_r^{r'} (s - r)^{2\alpha_2 -1} ds} \\
			&\le&
			C \sqrt{\big| (t-r')^{2\alpha_1-1} - (t-r)^{2\alpha_1 -1} \big| } ~(r'-r)^{\alpha_2}
			~\le~
			C(r'-r)^{2 \alpha' \wedge 1}.
		\end{eqnarray*}
	\end{remark}

	Recall that $\eta_n(s)$ is defined in \eqref{eq:def_eta_n},
	then by freezing the time in coefficient functions $(b, \sigma)$ (but not in $K_1, K_2$), and expanding $b$ and $\sigma$ in the space variable $x$,
	we obtain the following Milstein scheme for Equation \eqref{eq:defX}:
	\begin{eqnarray} \label{eq:MilsteinScheme}
		\Xb_t^n
		&=&
		X_0
		~+~
		\int_0^tK_1(t,s)\Big( b \big( \ens,\Xb_{\ens}^n \big)+ \nabla_x b \big( \ens, \Xb_{\ens}^n \big) \cdot A_s^{1,n}\Big)~ds \nonumber \\
		&&+\int_0^t K_2(t,s)\Big(\sigma \big( \ens, \Xb_{\ens}^n \big) + \nabla_x \sigma \big( \ens, \Xb_{\ens}^n \big) \cdot A_s^{n}\Big)~dW_s,
	\end{eqnarray}
	where
	$$
		A_s^{1,n} := \int_0^{\ens} \Big(K_2(s,r)-K_2(\ens,r)\Big)\sigma \big( \enr, \Xb_{\enr}^n \big)~dW_r,
	$$
	\begin{equation}\label{eq:defAn}
		A_s^{n}
		:=
		A_s^{1,n}
		+
		A_s^{2,n},
		~~\mbox{with}~~
		A_s^{2,n}
		:=
		\int_{\ens}^s K_2(s,r)\sigma \big( \enr, \Xb_{\enr}^n \big)~dW_r,
	\end{equation}
	and
	$$
		\nabla_x b(\cdot) \cdot A := \big( \langle \nabla_x b_i(\cdot),  A \rangle \big)_{1 \le i \le d}
		~~\mbox{and}~
		\nabla_x \sigma(\cdot) \cdot A := \big( \langle \nabla_x \sigma_{i,j} (\cdot), A \rangle \big)_{1 \le i, j \le d}.
	$$

	\begin{remark}
		Formally, the Milstein scheme   \eqref{eq:MilsteinScheme} is obtained by considering the first order Taylor expansion of $(b(t,x), \sigma(t,x))$ in  the space variable  $x$.
		Let us consider the points on the discrete-time grid $\pi_n = \{t_k ~: k =0, \dots, n \}$,
		then by \eqref{eq:defX}, one has
		\begin{align*}
			X_{t_{k+1}}-X_{t_k}
			=&
			\int_0^{t_k} \Big( K_1\big(t_{k+1},s\big) - K_1\big(t_k,s\big)\Big) b(s,X_s)~ds
			+
			\int_{t_k}^{ t_{k+1} } K_1\big( t_{k+1},s\big) b(s,X_s)~ds\\
			&+
			\int_0^{t_k}\Big( K_2\big(t_{k+1},s\big)- K_2\big(t_k,s\big)\Big) \sigma(s,X_s)~dW_s
			+
			\int_{t_k}^{t_{k+1}} K_2\big(t_{k+1},s\big) b(s,X_s)~dW_s,
		\end{align*}
		which, by Taylor expansion on $(b, \sigma)$, approximately equals to
		\begin{align*}
			&
			\int_0^{t_k} \Big( K_1(t_{k+1},s)-K_1(t_k,s)\Big)
			\Big(b(\ens,X_{\ens})+ \nabla_x b(\ens,X_{\ens}) \cdot  A^1_{s}\Big)~ds\\
			&+\int_{t_k}^{t_{k+1}}K_1(t_{k+1},s)\Big(b(\ens,X_{\ens})+ \nabla_x  b (\ens,X_{\ens}) \cdot A^1_{s}\Big)~ds\\
			&+\int_0^{t_k} \Big(K_2(t_{k+1},s)-K_2(t_k,s)\Big)\Big(\sigma(\ens,X_{\ens})+\nabla_x \sigma(\ens,X_{\ens}) \cdot A_s \Big)~
			dW_s\\
			&+\int_{t_k}^{t_{k+1}}K_2\big(t_{k+1},s\big)
			\Big(\sigma(\ens,X_{\ens})+ \nabla_x \sigma (\ens,X_{\ens}) \cdot A_s \Big)~dW_s,
		\end{align*}
		where
		$$
			A^1_s :=\int_0^{\ens} \!\!\!\! \Big(K_2(s,r)-K_2(\ens,r)\Big)\sigma(\eta_{n}(r),X_{\eta_{n}(r)})~dW_r,
			~~
			A_s := A^1_s+\int_{\ens}^s \!\!\!\!\! K_2(s,r)\sigma(\eta_{n}(r),X_{\eta_{n}(r)})~dW_r.
		$$
	\end{remark}

	\begin{remark} \label{rem:Milstein_time_freezing}
		In the Milstein scheme \eqref{eq:MilsteinScheme}, we do not freeze the second time variable $s$ for $K_i(t,s)$.
		In fact, in view of the last term in \eqref{eq:MilsteinScheme} and Condition \ref{eq:A4},
		replacing $K_2(t, s)$ by $K_2(t, \eta_n(s))$ would induce an $L^2$--error of the order
		$$
			\Big( \int_0^t  \big( K_2(t, s) - K_2(t, \eta_n(s) \big)^2 ds \Big)^{1/2}
			~\le~
			C \delta_n^{\alpha_2 \wedge 1},
		$$
		which is the same convergence as the Euler scheme (Theorem \ref{thm:EulerScheme}).
		In order to obtain an improvement of the convergence rate compared to the Euler scheme,
		we need to use $K_2(t,s)$ in place of $K_2(t, \eta_n(s))$ to construct the Milstein scheme.
		Notice that similar terms with freezing the second argument in $K$ have also appeared in the Hybrid scheme of \cite{BLP},
		where the authors suggested 
		$$
			\mbox{approximating}~
			\int_{-\infty}^t K(t,s) \sigma_s dW_s
			~~~\mbox{by}~~~
			\int_{t_0}^{t_1} K(t, \eta_n(s)) \sigma_{\eta_n(s)} dW_s + \int_{t_1}^t K(t, s) \sigma_{\eta_n(s)} dW_s,
		$$
		for some constants $t_0 < t_1 < t$.
		In particular, when the Hurst constant $\alpha$ is closed to $0$, the singularity problem could become important.
		With freezing the argument $s$ in $\sigma_s$ but not in $K(t,s)$ when $s \nearrow t$, 
		it has a better numerical performance.
		See also our numerical results in Section \ref{subsec:VOU}.
	\end{remark}
		
	\begin{remark} \label{rem:Milstein_implementation}
		$\mathrm{(i)}$ Let us consider the simulation of the Milstein scheme on the discrete-time grid $\{ t_k ~:k=0, \dots, n\}$,
		then the equation \eqref{eq:MilsteinScheme} and \eqref{eq:defAn} can be reduced to an induction system of finite number of random variables
		$\big\{ (B^k_i)_{i \in I_k}, ~k=0, 1, \dots, n \big\}$,
		where $B^k = (B^k_i)_{i \in I_k}$ is a function of $(W_s ~: s \in [0,t_k])$,
		which can be written as, for some functionals $f^k_{i,1}, f^k_{i,2}, f^k_{i,3}, f^k_{i,4}$,
		\begin{eqnarray} \label{eq:Bki}
			B^{k+1}_i
			&=&
			\int_{t_k}^{t_{k+1}} f^{k+1}_{i, 1} (B^k, s) ds
			+
			\int_{t_k}^{t_{k+1}} f^{k+1}_{i, 2}(B^k, s) \, dW_s \nonumber \\
			&&+
			\int_{t_k}^{t_{k+1}} \int_{t_k}^s  f^{k+1}_{i, 3}(B^k, r, s) \, dW_r \, ds
			+
			\int_{t_k}^{t_{k+1}} \int_{t_k}^s  f^{k+1}_{i, 4} (B^k, r, s) \, dW_r \, dW_s.
		\end{eqnarray}
		The challenge would be the simulation of the (correlated) double stochastic integrals
		$$
			\int_{t_k}^{t_{k+1}} \int_{t_k}^s  f^{k+1}_{i, 4} (B^k, r, s) \, dW_r\,  dW_s.
		$$
		In general, one may need to consider a finer discrete-time grid on $[t_k, t_{k+1}]$ to approximate the above integrals appearing
		in the induction expression of $B^{k+1}_i$.
		
		\vspace{0.5em}

		\noindent $\mathrm{(ii)}$ Nevertheless, in a first special case, 
		where $\sigma(t,x)$ is independent of $x$, so that $\nabla_x \sigma \equiv 0$,
		there is no double stochastic integral in the Milstein scheme \eqref{eq:MilsteinScheme} anymore.
		The problem reduces to the simulation of a fractional Brownian motion,
		which can be simulated exactly by computing the correlation of the increment of fractional Brownian motion.
		In a second special case, where $K_2 \equiv I_d$, 
		the double stochastic integral reduces to the form 
		$\int_{t_k}^{t_{k+1}}  (W_s - W_{t_k}) \, d W_s$,
		which can be simulated exactly when $d=1$, but is a L\'evy area when $d > 1$ as in the Milstein scheme for classical SDEs.
	\end{remark}

	\begin{theorem}\label{th:convMilstScheme}
		Let Assumptions \ref{assum:main} and \ref{assum:main2} hold true, and recall the constant $\alpha'> 0$ defined in Assumption \ref{assum:main2}.
		
		\vspace{0.5em}
		
		\noindent $(i)$ Let $p \geq \max\big( \frac{\beta_1}{\beta_1-1}, \frac{2\beta_2}{\beta_2 -1} \big)$.
		There exists a constant $C_p \in(0,\infty)$ depending only on $T$, $d$, $p$ and $\beta_1$, $\beta_2$, $C$ in Assumptions \ref{assum:main} and \ref{assum:main2} such that,
		for all $s, t \in [0,T]$ and $n \ge 1$,
		\begin{equation} \label{eq:Rate_MilstScheme}
			\EE \Big[ \big| \Xb_t^n-\Xb_s^n \big|^p \Big]  
			\leq 
			C_p  \big(1+\EE \big[|X_0|^p \big] \big) |t-s|^{p (\alpha' \wedge 1)}
			~~\mbox{and}~~
			\EE \Big[ \big| \Xb^n_t-X_t \big|^p \Big]  \leq C_p  \big(1+\EE \big[|X_0|^p \big] \big) \dn^{p (2\alpha' \wedge 1)}.
		\end{equation}
		
		\noindent $(ii)$ Let, in addition, $p> (2\alpha'\wedge 1)^{-1}$.
		Then for all $\varepsilon\in(\frac{1}{p}, 2 \alpha' \wedge 1)$, there exists $C_{p, \varepsilon} \in (0, \infty)$ such that
		$$
			\Big(\EE \Big[ \sup_{t\in[0,T]} \big| \Xb^n_t-X_t \big|^p \Big] \Big)^{\frac 1p}
			\leq
			C_{p, \varepsilon} \big(1+\EE \big[|X_0|^p \big] \big)^{\frac{1}{p}} \, \dn^{ (2\alpha'\wedge 1) - \varepsilon}.
		$$
	\end{theorem}

	\begin{remark} 
		When $K_1 = K_2 \equiv I_d$, the Volterra equation \eqref{eq:defX} degenerates into the standard SDE,
		and Assumptions \ref{assum:main} and \ref{assum:main2} hold with $\alpha_1 = 1$, $\alpha_2 = \frac12$ so that $\alpha = \alpha'  = \frac12$.
		In this case, one has $A^{1,n}_s \equiv 0$ and hence the term containing $\nabla_{x}b$ in \eqref{eq:MilsteinScheme} disappears.
		Consequently, our Milstein scheme \eqref{eq:MilsteinScheme} turns to be exactly the same as the Milstein scheme for standard SDEs studied in the literature.
		Moreover, the convergence rate \eqref{eq:Rate_MilstScheme} is also the same as that of the Milstein scheme for standard SDEs (see e.g. \cite[Chapter 7]{graham2013stochastic}).

		However, for the rate under the uniform convergence norm,
		it is less general due to the use of Garsia-Rodemich-Rumsey lemma in our technical proof.
	\end{remark}

\subsection{The (Multilevel) Monte-Carlo method and complexity analysis}
\label{subsec:MLMC}

	Let $f: C([0,T], \R^d) \to \R$ be a functional, Lipschitz under the uniform convergence norm. 
	We aim at estimating the expectation value:
	$$
		m 
		~:=~
		\E \big[ f (X_{\cdot}) \big].
	$$
	Based on $N$ copies of simulations 
	$\big\{ (X^{n,i}_{t_k})_{k = 0,1, \dots, n}, i = 1, \dots N \big\}$ of $X$
	on the  discrete-time grid $\pi_n = \{t_k ~: k=0, 1, \dots, n\}$,
	we can use linear interpolation to obtain $N$ continuous path $\Xh^{n,i}$ on $[0,T]$,
	and then obtain the Monte-Carlo estimator
	$$
		\widehat m^n_N 
		~:=~ 
		\frac 1N \sum_{i=1}^N f( \Xh^{n,i}_{\cdot}).
	$$
	Given $\eps > 0$, we will compute the number of operations a computer must perform to achieve an error of order $O(\eps)$ between $m$ and the corresponding Monte-Carlo estimator, such as $\widehat m^n_N$.
	The number of such operations is called computational cost or complexity of the algorithm.
	We will first study the Euler scheme  \eqref{eq:EulerScheme},
	and then, based on the convergence rate results for the Euler scheme,
	study the corresponding Multilevel Monte-Carlo (MLMC) method.

	\vspace{0.5em}

	Let us assume all the conditions in Theorem \ref{thm:EulerScheme}. 

\paragraph{The complexity and error analysis for the Euler scheme.}

	To simulate a path of solution $X^n$ to the Euler scheme \eqref{eq:EulerScheme} on the discrete-time grid $\pi_n = \{t_k ~: k=0, 1, \dots, n\}$,
	one needs to simulate $n$ increments of the Brownian motion $(\Delta W_k)_{k=1, \dots, n}$ and take the sum $O(n^2)$ times
	(see Remark \ref{rem:Euler_Implementation}).
	The complexity to simulate $N$ copies of paths of $(X^n_{t_k})_{k = 1, \dots, n}$ will then be
	$O( N n^2)$.
	
	\vspace{0.5em}

	Now to achieve an error of order $O(\eps)$ for any $\eps > 0$, we need to let both the discretization error and statistical error be of order $O(\eps)$. To control the statistical error, it is clear that
	one needs to set $N = O(\eps^{-2})$. As for the discretization error, Theorem \ref{thm:EulerScheme} implies that, with $\alpha_{\circ} > 0$ satisfying
	$$
		\begin{cases}
			\alpha_{\circ} = \alpha \wedge 1, &\mbox{if}~f(X_{\cdot})~\mbox{depends on}~(X_t)_{t \in \mathbb{T}}~\mbox{for a finite subset}~\mathbb{T}~\mbox{of}~[0,T],\\
			\alpha_{\circ} \in (0, \alpha \wedge 1), & \mbox{otherwise},
		\end{cases}
	$$
	one needs to set $n = O(\eps^{- \alpha_{\circ}^{-1}})$. 
	We summarize the previous discussion in the following proposition.
	\begin{proposition}
		Denote by $ \left(C_1(\eps)\right)_{\varepsilon>0}$ the complexity of the Monte-Carlo estimation $\widehat m^n_N$ of $m$ by the Euler scheme, then
		\begin{equation} \label{eq:Complexity_Euler}
			C_1(\eps)  \le C \eps^{- 2 - 2 \alpha_{\circ}^{-1} },
			~~\mbox{for some constant}~C>0.
		\end{equation}
	\end{proposition}
	
	\begin{remark}
		For Milstein scheme, if the terms $B^{k+1}_i$ in \eqref{eq:Bki} can be simulated with computation effort $O(1)$
		(e.g. when $\nabla_{x} \sigma \equiv 0$, or $K_2 \equiv I_d$, see Remark \ref{rem:Milstein_implementation}),
		the complexity to simulate $N$ copies of paths $(\Xb^n_{t_k})_{k=1, \cdots, n}$ of the Milstein scheme is also of order $O(N n^2)$.
		Then using the same argument for Euler scheme, together with the convergence rate of Milstein scheme in Theorem \ref{th:convMilstScheme}, 
		one can deduce that the complexity of the Milstein scheme is bounded by $C \eps^{-2 - 2/\alpha'_{\circ}}$ with 
	$$
		\begin{cases}
			\alpha'_{\circ}= 2 \alpha' \wedge 1, &\mbox{if}~f(X_{\cdot})~\mbox{depends on}~(X_t)_{t \in \mathbb{T}}~\mbox{for a finite subset}~\mathbb{T}~\mbox{of}~[0,T],\\
			\alpha'_{\circ} \in (0, 2\alpha' \wedge 1), & \mbox{otherwise}.
		\end{cases}
	$$
	When the terms $B^{k+1}_i$ in \eqref{eq:Bki} cannot be simulated exactly,
	its complexity analysis would depend on how to approximate $B^{k+1}_i$ terms in practice and 
	the corresponding approximation error. 
	\end{remark}

\paragraph{The MLMC method.}

	We adapt the MLMC method of Giles \cite{giles2008multilevel} to our context.
	Although the statement of \cite[Theorem 3.1]{giles2008multilevel} does not apply directly here, the arguments stay the same.
	Let $M \ge 2$ be some positive integer, set $n_{\ell} := M^{\ell}$ so that $h_{\ell} := M^{-\ell}T$ for all $\ell \in \N$.
	Let $X^{n_{\ell}}$ denote the numerical solution of the Euler scheme using the uniform discretization with step size $\Delta t := h_{\ell}$,
	and
	$$
		\Ph := f(X_{\cdot}),
		~~~
		\Ph_\ell := f(X^{n_{\ell}}_{\cdot}),~ \ell \ge 0,
	$$
	so that
	\begin{equation} \label{eq:def_P_l}
		\big| \E \big[ \Ph_{\ell} \big] - \E \big[ \Ph \big]  \big| \le C h^{\alpha_{\circ}}_{\ell},
		~\mbox{and}~
		\E \big[ \big| \Ph_{\ell}  - \Ph_{\ell - 1} \big|^2 \big] \le C h^{2 \alpha_{\circ}}_{\ell},
		~\ell \ge 1,
	\end{equation}
	for some constant independent of $\ell$.
	Notice that
	$$
		\E \big[ \Ph_L \big]
		~=~
		\E \big[ \Ph_0 \big]
		+ \sum_{\ell = 1}^{L} \E \big[ \Ph_{\ell} - \Ph_{\ell - 1} \big].
	$$
	To estimate $\E [\Ph_0]$, we simulate $N_0$ i.i.d. copies $(X^{n_0, i})_{i=1, \dots, N_0}$ of $X^{n_0}$ and use the estimator
	$$
		\Yh_0 ~:=~ \frac{1}{N_0} \sum_{i=1}^{N_0} f( \Xh^{n_0, i}_{\cdot}).
	$$
	To estimate $\E \big[ \Ph_{\ell} - \Ph_{\ell -1} \big]$ for $\ell \ge 1$,
	we simulate $N_{\ell}$ i.i.d. copies $(X^{n_{\ell},i}, X^{n_{\ell-1},i})$ of $(X^{n_{\ell}}, X^{n_{\ell-1}})$ and use the estimator
	\begin{equation} \label{eq:estimate_Yl}
		\Yh_{\ell} ~:=~ \frac{1}{N_{\ell}} \sum_{i=1}^{N_{\ell}} \Big( f( \Xh^{n_{\ell}, i}_{\cdot}) - f( \Xh^{n_{\ell-1}, i}_{\cdot}) \Big).
	\end{equation}
	Then our MLMC estimator for $\E [ f(X_{\cdot}) ]$ is given by
	$$
		\Yh ~:=~ \sum_{\ell = 0}^{L} \Yh_{\ell},
		~~\mbox{whose numerical computation effort is of order}~
		 \sum_{\ell = 0}^{L} N_{\ell} O( h_{\ell}^{-2}).
	$$
	To meet the error level $\eps > 0$,
	one can set $L \ge 1$ and $N_{\ell} \ge 1$ such that
	$$
		h^{\alpha_{\circ}}_L = O(\eps),
		~~~
		N_0^{-1} = O(\eps^{-2})
		~~\mbox{and}~~
		N_{\ell}^{-1}  \mathrm{Var}[  \Ph_{\ell} - \Ph_{\ell - 1} ] \le C N_{\ell}^{-1} h^{2 \alpha_{\circ}}_{\ell} = O(\eps^{2}),
		~\mbox{for}~\ell \ge 1,
	$$
	where the bound of $ \mathrm{Var}[  \Ph_{\ell} - \Ph_{\ell - 1} ]$ follows from \eqref{eq:def_P_l}.
	By direct computations, one obtains that the complexity of the MLMC estimator $\Yh$ is bounded, for some constant $C$ independent of $\eps$, by
	\begin{equation*} 
		C \sum_{\ell = 0}^{L} N_{\ell} h_{\ell}^{-2}
		\le
		C \sum_{\ell=0}^{L} \eps^{-2} h_{\ell}^{2(\alpha_{\circ} -1)}
		=
		\begin{cases}
			C | \log(\eps) | \eps^{-2}, &~\mbox{if}~ \alpha_{\circ} = 1,\\
			C \eps^{- 2 \alpha_{\circ}^{-1}  }, &~\mbox{if}~ \alpha_{\circ} < 1.
		\end{cases}
	\end{equation*}
	We summarize the previous discussion in the following proposition.
	\begin{proposition}
		Denote by $ \left(C_2(\eps)\right)_{\varepsilon>0}$ the complexity of the Multilevel Monte-Carlo estimation $\widehat Y$ of $m$, then 
		\begin{equation} \label{eq:Complexity_MLMC}
			C_2(\eps) 
			\le
			\begin{cases}
				C  | \log(\eps) | \eps^{-2}, &~\mbox{if}~ \alpha_{\circ} = 1,\\
				C \eps^{- 2 \alpha_{\circ}^{-1}}, &~\mbox{if}~ \alpha_{\circ} < 1,
			\end{cases}
						~~\mbox{for some constant}~C>0.
		\end{equation}
	\end{proposition}
	The complexity result in \eqref{eq:Complexity_MLMC} consists of a significant improvement compared to $C_1(\eps)$ in \eqref{eq:Complexity_Euler}.

	\begin{remark}
		The MLMC method for the Milstein scheme \eqref{eq:MilsteinScheme}-\eqref{eq:defAn} seems also to be very interesting.
		Nevertheless, due to the implementation problems (see e.g. Remark \ref{rem:Milstein_implementation}), it seems less clear how to introduce an implementable algorithm.
		A possible approach would be extending the (antithetic) Milstein MLMC method of classical SDEs such as in \cite{giles2008improved, GilesSzpruch} to our context.
		We would like to leave this for future research.
	\end{remark}

\section{Numerical examples}
\label{sec:examples}

	We will implement the simulation methods introduced above, including the Euler scheme, Milstein scheme and the MLMC method,
	on three Volterra equations, including two affine equations and another equation from statistical mechanics.
	For the two affine equations, we are able to compute explicitly the reference values for comparison in some cases.
	
\subsection{A Volterra Ornstein-Uhlenbeck equation}
\label{subsec:VOU}

	We first consider a one-dimensional Volterra Ornstein-Uhlenbeck equation which is a special case of  Equation \eqref{eq:defX},
	where $b(x)=b_0+b_1x,~\sigma(x)\equiv\sigma_0$ for some constants $b_0,~b_1,~\sigma_0 \in \mathbb{R}$ and kernel $K(t,s) \equiv K(t-s)=(t-s)^{H-\frac{1}{2}}/\Gamma({H+\frac{1}{2}})$, that is,
	\begin{equation*}
		X_t=x_{0}+\int_0^t K(t-s)\, (b_0+b_1X_s)~ds+\int_0^t K(t-s)\, \sigma_0~dW_s.
	\end{equation*}
	The above Volterra Ornstein-Uhlenbeck equation appears naturally in many applications, for instance  in turbulence \cite{Chevillard}, or as a non-Markovian Langevin equation in statistical mechanics (see \cite{JaksicPilletLetter,Lutz}),
	and its solution can be computed explicitly (see e.g. \cite{jaber2019affine}):  
	\begin{equation*}
		X_t = \Big( 1-\int_0^t R_{b_1}(s) ~ds \Big) x_0 +b_0\int_0^tE_{b_1}(s)~ds +\sigma_0\int_0^t E_{b_1}(t-s)~dW_s,
	\end{equation*}
	where
	$$
		R_{b_1}(s)=-b_1s^{H-\frac{1}{2}}\sum_{n=0}^{\infty}\frac{(b_1 s^{H+\frac{1}{2}})^n}{\Gamma((n+1)(H+\frac{1}{2}))},
		~\mbox{and}~
		E_{b_1}(s)=s^{H-\frac{1}{2}}\sum_{n=0}^{\infty}\frac{(b_1 s^{H+\frac{1}{2}})^n}{\Gamma((n+1)(H+\frac{1}{2}))}.
	$$
 	Then $X_T$ is a Gaussian random variable, with
	$$
		\EE(X_T)= \Big(1-\int_0^T  R_{b_1}(s)~ds \Big)x_0+b_0\int_0^TE_{b_1}(s)~ds,
		~~\mbox{and}~
		\mbox{Var}(X_T)=\sigma_0^2\int_0^T E_{b_1}(T-s)^2~ds.
	$$
	We aim at estimating $\EE[(X_T-1)_+]$ by the methods introduced previously, namely the Euler scheme, Milstein scheme and the MLMC method.
	Notice that $\sigma(x) \equiv \sigma_0 $ is a constant, and hence the Milstein scheme can be easily implemented (see Remark \ref{rem:Euler_Implementation}).
	Let us choose parameters $x_0=1,~b_0=1$, $b_1=-0.5$, $T=1,~\sigma_0=0.2$. For these parameter values, since $X_{T}$ is Gaussian, one can evaluate $\EE[(X_T-1)_+]$ theoretically with arbitrary accuracy.  
	When $H= 0.1$, one has $\EE[(X_T - 1)_+] \approx 0.3978$;
	when $H=0.25$, one has $\EE[(X_T-1)_+] \approx 0.397202$;
	and  when $H =0.75$, one obtains $\EE[(X_T-1)_+] \approx 0.373444$, with an error smaller than $10^{-6}$ in all cases. 
	These will serve as reference values for our numerical tests. 

	\vspace{0.5em}

	For both Euler scheme and Milstein scheme, we use the uniform discretisation, with time step $\Delta t = T/n$. 
	We will test different values on the time discretisation parameter $n$, and for each test, we simulate $N = 10^4$ copies of $(X^n_T - 1)_+$.
	The statistical error is given by $\widehat \sigma_N /\sqrt{N}$, where $\widehat \sigma_N^2$ is the empirical variance of $(X^n_T - 1)_+$.
	
	\vspace{0.5em}
	
	For the implementation of the MLMC method (described in Section \ref{subsec:MLMC}), we choose $M = 4$ so that $n_{\ell} = 4^{\ell}$ and $h_{\ell} = T 4^{-\ell}$.
	Then for each (MSE) error level $\eps > 0$,
	we choose the maximum level $L$ and simulation numbers $N_{\ell}$ for $\ell = 0, \dots, L$, so that the discretization error and statistical error are both bounded by $\eps/\sqrt{2}$.
	Following the arguments in \cite{giles2008multilevel}, we choose the maximum level $L$ as follows:
	assume that (recall Theorem \ref{thm:EulerScheme}), for some constant $C> 0$,
	$$
		\E \big[\Ph -  \Ph_\ell \big]
		~\approx~
		C 4^{-\ell (\alpha \wedge 1)}, 
		~~\mbox{with}~ \alpha = H,
		~\Ph = (X_T - 1)_+, ~~\Ph_{\ell} = (X^{n_{\ell}} - 1)_+.
	$$
	Then one has
	$$
		\EE[ \Ph_\ell-\Ph_{\ell-1} ]
		~\approx~
		(4^{\alpha \wedge 1} -1)C 4^{-\ell(\alpha \wedge 1)}
		~\approx~
		(4^{\alpha \wedge 1}-1)\EE[\widehat{P}-\widehat{P}_\ell].
	$$
	Recall that $\Yh_{\ell}$ is the estimation of $\E \big[ \Ph_\ell-\Ph_{\ell-1} \big]$,
	for each numerical experiment,
	the constant $L$ is chosen as the smallest number satisfying
	\begin{equation} \label{eq:stopping_criterion}
		\max \big\{
		4^{-(\alpha \wedge 1)} \big|\Yh_{L-1}|,| \Yh_L \big| 
		\big\}
		< 
		(4^{(\alpha \wedge 1 )}-1) \frac{\varepsilon}{\sqrt{2}},
	\end{equation}
	in order to ensure (empirically) that the discretization error $\E \big[\Ph -  \Ph_L \big] \le \varepsilon/\sqrt{2}$.
	To choose $N_{\ell}$, we let $V_{\ell}$ denote the empirical variance of $\Ph_{\ell} - \Ph_{\ell -1}$ when $\ell \ge 0$ (with $\Ph_{-1} = 0$), the statistical error is measured by
	$$
		\sqrt{ \sum_{\ell=0}^L N_\ell^{-1}V_\ell}.
	$$
	Recall that the computation effort to estimate $\E \big[ \Ph_{\ell} - \Ph_{\ell -1} \big]$ is proportional to $N_{\ell} h_{\ell}^{-2}$,
	we choose $(N_{\ell} ~:\ell = 0, \dots, L)$ which minimizes the computational effort, under the statistical error constraint, that is,
	$$  
		\min_{N_{\ell}}
		\sum_{\ell=0}^L N_{\ell} h_\ell^{-2},
		~~\mbox{subject to}~
		\sum_{\ell=0}^L \frac{V_\ell}{N_\ell}\leq \frac{\varepsilon^2}{2}.
	$$
	This leads to the optimal choice of $N_{\ell}$:
	\begin{equation} \label{eq:estimate_Nl}
		N_{\ell} 
		~\approx~
		2\varepsilon^{-2}\sqrt{V_{\ell}h_\ell^2} \sum\limits_{m=0}^L\sqrt{V_m/h_m^2}.
	\end{equation}

	\begin{remark} \label{rem:MLMC_pseudocode}
	To summarize, the pseudo-code for MLMC method is given as follows:
		
	\vspace{0.2em}
		
	\noindent Step 1. Initiate with $L = 0$.
	
	\vspace{0.2em}
	
	\noindent Step 2. Set the initial value $N_L = 100$ and simulate $N_L$ copies of  $\Ph_L - \Ph_{L-1}$, and compute its empirical variance $V_L$.
	
	\vspace{0.2em}
		
	\noindent Step 3. Update the values $N_{\ell}$ for $\ell = 0, 1, \cdots, L$, using \eqref{eq:estimate_Nl}.
	
	\vspace{0.2em}
	
	\noindent Step 4. Simulate extra samples according to the updated values $(N_{\ell})_{\ell = 0 ,1 ,\cdots, L}$, 
	and update the estimators $(\Yh_{\ell})_{\ell = 0, 1, \cdots, L}$ by \eqref{eq:estimate_Yl}, 
	and compute the corresponding empirical variances $(V_{\ell})_{\ell = 0, 1, \cdots, L}$.
	
	\vspace{0.2em}
	
	\noindent Step 5. If $L < 2$ or \eqref{eq:stopping_criterion} is not true, set $L = L +1$ and go to Step 2.
	Otherwise, quit the program and return the estimation $\sum_{\ell = 0}^L \Yh_{\ell}$.
	\end{remark}

	The simulations are implemented in Python, and the simulation results, together with the reference values, are given in Tables  \ref{tab:affine_0}, \ref{tab:affine_1} and \ref{tab:affine_2}.
	We can observe the convergence of the Euler scheme and the Milstein scheme, as $n$ increases.
	When $H = 0.1$, the Euler scheme does not perform as well as the cases $H=0.25$ and $H=0.75$, but the Milstein scheme performs relatively well for all the cases.
	This should be due to the no-time-freezing on $K$ in the Milstein scheme, which is also consistent with the results in \cite{BLP} for their Hybrid scheme (see also Remark \ref{rem:Milstein_time_freezing}).
	The MLMC method converges well when $\varepsilon$ decreases, and it performs better (in terms of computation time) for the regular case when $H = 0.75$ than the singular case when $H=0.1$ or $H=0.25$.
	We also provide the computation time for different examples, which is however only indicative.
	Indeed, the computation time depends not only on the complexity of the scheme, but essentially on the way the schemes are implemented in a concrete programming language.
	For example, for the Euler scheme, when the number of simulations are given, one can use array operations in the simulation and computation of the estimator in Python.
	However, for the implementation of MLMC method, one needs to use loop operations (see Remark \ref{rem:MLMC_pseudocode}) which may take more time for the same complexity of computation in Python.

\begin{table}[!htbp]
\centering
\begin{tabular}{ | c || c | c | c |}
	\hline
	& Mean Value & Statistical Error  & Computation time(s)  \\ \hline
	Reference value & 0.397800& - & -  \\ \hline
	
	Euler Scheme {(n=8)} & 0.362125 & 0.001613 & 0.733955 \\ \hline
	
	Euler Scheme {(n=20)} & 0.374513 & 0.001707 & 1.632285\\ \hline
	 
	Euler Scheme {(n=40)} & 0.377491 & 0.001797 & 3.131211 \\ \hline
	
	Euler Scheme {(n=80)} & 0.385454 & 0.001896 & 6.086132 \\ \hline
	
	Milstein Scheme {(n=8)} & 0.399300 & 0.001690 & 5.753037 \\ \hline
	
	Milstein Scheme {(n=20)} & 0.392401 & 0.001684 & 34.479045  \\ \hline
	
	Milstein Scheme {(n=40)} &0.394541 & 0.001699 & 343.109104 \\ \hline
	
	MLMC {($\varepsilon=0.08$)} & 0.386223 & 0.009216 & 1.830342 \\ \hline
	
	MLMC {($\varepsilon=0.05$)} & 0.406156 & 0.00738 & 9.688824 \\ \hline
	
	MLMC {($\varepsilon=0.03$)} &0.393942 & 0.005917 & 204.701933 \\ \hline	
	\end{tabular}
	\caption{Numerical estimation of $\EE [(X_T-1)_+]$ under Volterra O-U model with $H=0.1$.}
	\label{tab:affine_0}
\end{table}

\begin{table}[!htbp]
\centering
\begin{tabular}{ | c || c | c | c |}
	\hline
	& Mean Value & Statistical Error  & Computation time(s)  \\ \hline
	Reference value & 0.397202 & - & -  \\ \hline
	
	Euler Scheme {(n=8)} & 0.387355 & 0.001628 & 0.681871 \\ \hline
	
	Euler Scheme {(n=20)} & 0.391123 & 0.001676 & 1.481485 \\ \hline
	 
	Euler Scheme {(n=40)} & 0.393647 & 0.001691 & 2.724405 \\ \hline
	
	Euler Scheme {(n=80)} & 0.395506 & 0.001718 & 5.356922 \\ \hline
	
	Milstein Scheme {(n=8)} & 0.406978 & 0.001670 & 5.532334 \\ \hline
	
	Milstein Scheme {(n=20)} & 0.399021 & 0.001649 & 30.704011  \\ \hline
	
	Milstein Scheme {(n=40)} &0.398021 & 0.001642 & 339.831242 \\ \hline
	
	MLMC {($\varepsilon=0.01$)} & 0.395004 & 0.003565 & 1.192638 \\ \hline
	
	MLMC {($\varepsilon=0.007$)} & 0.396933 & 0.002435 & 5.099278 \\ \hline
	
	MLMC {($\varepsilon=0.005$)} &0.396543 & 0.001679 & 26.984565 \\ \hline	
	\end{tabular}
	\caption{Numerical estimation of $\EE [(X_T-1)_+]$ under Volterra O-U model with $H=0.25$.}
	\label{tab:affine_1}
\end{table}

\begin{table}[!htbp]
	\centering
	\begin{tabular}{ | c || c | c | c |}
	\hline
	& Mean Value & Statistical error  & Computation time(s)  \\ \hline
	
	Reference value &0.373444&-&- \\ \hline

	Euler Scheme {(n=8)}    & 0.393690 & 0.001582 & 0.677893 \\ \hline

	Euler Scheme {(n=20)}    & 0.382303 & 0.001544 & 1.453350 \\ \hline

	Euler Scheme {(n=40)}    &	0.377869 & 0.001527 & 2.673218 \\ \hline
		
	Euler Scheme {(n=80)}    &	0.376531 & 0.001518 & 5.246249 \\ \hline

	Milstein Scheme (n=8)    &0.382396 & 0.001698 & 4.926303 \\ \hline

	Milstein Scheme (n=20)    & 0.378372 & 0.001693 & 17.486181 \\ \hline

	Milstein Scheme (n=40)    & 0.375271 & 0.001670 & 141.646946 \\ \hline

	MLMC ($\varepsilon=0.01$) &0.382145 & 0.004835 & 0.389519 \\ \hline

	MLMC ($\varepsilon=0.007$) &0.370043 & 0.003060 & 1.034552 \\ \hline

	MLMC ($\varepsilon=0.005$) &0.374237 & 0.002400 & 1.544744 \\ \hline
	\end{tabular}
	\caption{Numerical estimation of $\EE [(X_T-1)_+]$ under Volterra O-U model with $H=0.75$.}
	\label{tab:affine_2}
\end{table}

\subsection{ A Volterra equation arising in statistical mechanics}

	We  consider next an example of stochastic Volterra equation which originates from works in statistical mechanics (see e.g.  Jak{\v{s}}i{\'c} and Pillet \cite{JaksicPillet}). 
	Namely, let $(q,p)$ denotes the couple position-speed of a particle evolving in a heat bath in a one-dimensional space. Then, adapting slightly the equations of Hamiltonian mechanics, 
  the following equations can be derived
	\begin{equation*}
		\dot q_{t} = p_{t},
		~~~~~
		\dot p_{t} = V'(q_{t}) \Big(-1 - \lambda^2  \int_{0}^t K(t-s)^2~ V(q_{s})\, ds - \lambda  \int_{0}^t K(t-s) \, V(q_{s}) \, dW_{s} \Big),
	\end{equation*}
	for some Brownian motion $W$. 
	Now, provided that $V'$ does not vanish, rewriting the whole system in an augmented form with 
	$$
		X_t = \Big( q_{t}, ~p_t,  ~\frac{\dot p_{t}}{V'(q_{t})} \Big),
	$$
	yields the following stochastic Volterra equation
	\begin{equation}\label{eq:dissip}
		\begin{cases}
		&X^{(1)}_{t} = X^{(1)}_0 + \int_0^t X^{(2)}_{s}\, ds, \\
		&X^{(2)}_{t} = X^{(2)}_0 + \int_0^t X^{(3)}_{s}\, V'(X^{(1)}_{s})\, ds, \\
		&X^{(3)}_{t} =  -1 - \lambda^2  \int_{0}^t K(t-s)^2~ V(X^{(1)}_{s}) \, ds - \lambda  \int_{0}^t K(t-s) \, V(X^{(1)}_{s}) \, dW_{s} ~.
		\end{cases}
	\end{equation}
	
	We simulate the above Volterra equation with the following parameters: $X^{(1)}_0 = X^{(2)}_0 = 0$, $\lambda = 3$, $T =2$, 
	$K(t) = t^{H-1/2}$  with $H= 0.3$ or $H=0.7$, 
	and $V(x) := x + \alpha \cos(x) $ (which corresponds to the physical confining potential $\mathcal{V}(x) = \frac{x^2}{2} + \alpha \sin(x)$) with $\alpha = 0.1$.
	Notice that $X^{(1)}_t$ represents the position of the particle, $X^{(2)}$ represents the speed and $X^{(3)}$ the acceleration of the particle at time $t$.
	In Figure \ref{Fig1}, we provide a simulation of the paths of $(X, W)$ on time interval $[0,4]$ with different constants $H$, by using the Euler scheme with time steps number $n = 1000$.
	 As expected, the roughness of $X^{(3)}$ increases as $H$ decreases (in fact the H\"older regularity of $X^{(3)}$ is almost $H$). 
	As expected for Langevin-type dynamics, one can observe a mean-reversion phenomena in the simulation of 
	$(X^{(1)}, X^{(2)}, X^{(3)})$ in Figure \ref{Fig1}.

	\vspace{0.5em}
	
	We provide some estimation by Monte Carlo simulation on the 1st, 2nd and 3rd moments of the position, i.e. $\E[X^{(1)}_T]$, $\E \big[ (X^{(1)}_T)^2 \big]$ and $\E \big[ (X^{(1)}_T)^3 \big]$ for $T=2$,
	based on the Euler scheme and the corresponding MLMC method.
	For Euler scheme, each estimation is obtained with $N=10^4$ simulated copies of $X_{T}$, and 	the statistical errors are given by $\widehat \sigma_N /\sqrt{N}$, where $\widehat \sigma_N^2$ is the empirical variance of $(X^{(1)})^k$ for $k=1,2,3$ according to the case.
	We try different number of time steps $n$.	
	For the MLMC method, we use the same parameter $M=4$ and follow the same procedure as in Section \ref{subsec:VOU} to compute the estimations with different MSE error $\eps > 0$.

	\vspace{0.5em}

	The simulation results are provided in Tables \ref{tab:PositionSpeed_1}  and \ref{tab:PositionSpeed_3}.
	From the simulation results of the MLMC method with $\eps = 0.05$,
	for $H=0.3$, the estimated mean value, variance and skewness of $X^{(1)}_T$ are approximately $(0.81069, 0.17045, -0.18988)$,
	and for $H = 0.7$, the corresponding values are approximately $(-1.29678, 0.18895, -0.62810)$.
	It seems that the estimations from the Euler scheme simulation converge to these values as $n$ increases.



\begin{figure} 
	\begin{center}
	\includegraphics[width=14cm,height=9cm]{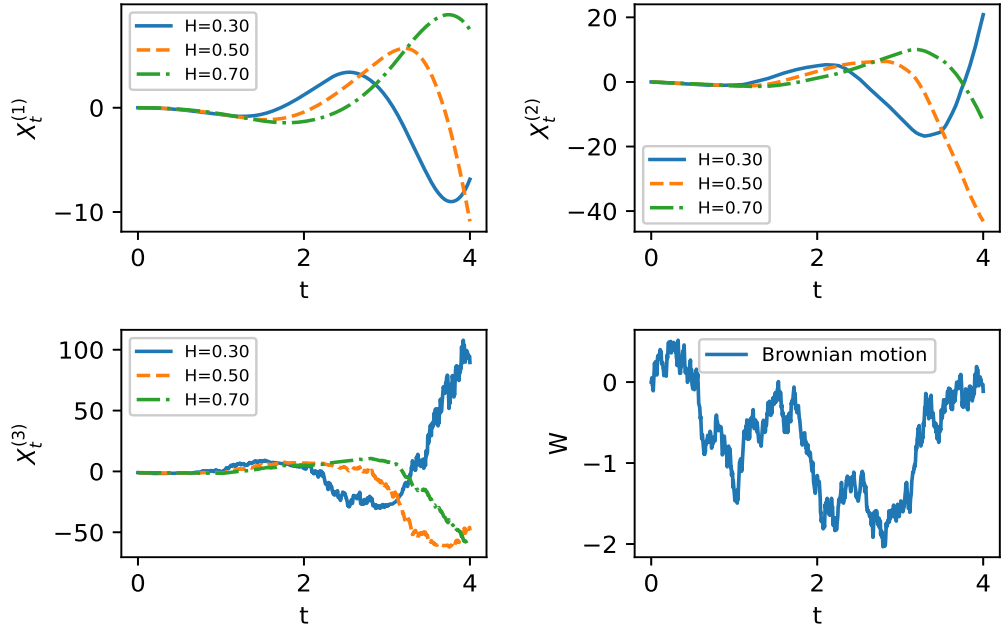}
	\end{center}
	\caption{A simulation of paths of $(X^{(1)}, X^{(2)}, X^{(3)}, W)$ on time interval $[0,4]$.}
	\label{Fig1}
\end{figure}

\begin{table}[!htbp]
\centering
	\begin{tabular}{ | c || c | c || c |c||c|c|}
	\hline
	&Mean($k=1$) & S. Error  &Mean($k=2$) & S. Error  & Mean($k=3$) &S. Error\\ \hline
	
	Euler Sch. {(n=100)} &
	0.678132 & 0.005183
	&  0.728080 & 0.007653
	&0.844422 & 0.013723 
	 \\ \hline

	Euler Sch. {(n=500)} &
	0.783380 & 0.004610
	&0.826746 & 0.007585 
	&0.948628 & 0.013016
	\\ \hline	
	
	Euler Sch. {(n=1000)} &
	0.790071 & 0.004441
	&0.825749 & 0.007319
	&0.947221 & 0.012555
	\\ \hline

	MLMC {($\varepsilon=0.1$)} & 
	0.775920 & 0.025893
	&  0.829574 & 0.041700
	&0.969688 & 0.039718
	 \\ \hline	
	 
	MLMC {($\varepsilon=0.07$)} & 
0.791636 & 0.015919 
	&  0.807127 & 0.027351
	&0.926160 & 0.026580
	 \\ \hline
	
	MLMC {($\varepsilon=0.05$)} &
	0.810689 & 0.013005
	& 0.827675 & 0.021849
	&0.934001 & 0.024089
	\\ \hline
	
	\end{tabular}
	\caption{Monte Carlo estimation of $\E  \big[ (X_T^{(1)})^k \big]$, with $H=0.3$ and $T=2.0$.}
	\label{tab:PositionSpeed_1}
\end{table}

\begin{table}[!htbp]
\centering
	\begin{tabular}{ | c || c | c || c |c||c|c|}
	\hline
	&Mean($k=1$) & S. Error  &Mean($k=2$) & S. Error  & Mean($k=3$) &S. Error\\ \hline
	
	Euler Sch. {(n=100)} &
	-1.382149 & 0.004254
	&  2.085036 & 0.012392 
	&-3.447007 & 0.031222
	 \\ \hline

	Euler Sch. {(n=500)} &
-1.315642 & 0.004440
	&1.927117 & 0.012459
	&-3.104342 & 0.031140
	\\ \hline	
	
	Euler Sch. {(n=1000)} &
-1.311788 & 0.004525 
	&1.882418 & 0.012361 
	&-2.992574 & 0.029891
	\\ \hline

	MLMC {($\varepsilon=0.1$)} & 
	-1.341978 & 0.033149
	&  1.939615 & 0.055981
	&-3.154188 & 0.053123
	 \\ \hline	
	
	MLMC {($\varepsilon=0.07$)} & 
	-1.346259 & 0.029277
	& 1.914329 & 0.034501
	&-2.961424 & 0.040303
	 \\ \hline
	
	MLMC {($\varepsilon=0.05$)} &
	-1.296778 & 0.024512
	&1.870580 & 0.027123
	&-2.967358 & 0.029538 
	\\ \hline
	
	\end{tabular}
	\caption{Monte Carlo estimation of $\E  \big[ (X_T^{(1)})^k \big]$, with $H=0.7$ and $T=2.0$.}
	\label{tab:PositionSpeed_3}
\end{table}

\subsection{The Rough Heston model }

	We  consider next the so-called rough Heston model introduced by El Euch and Rosenbaum \cite{ElEuchRosenbaum} (see also \cite{Jacquier} for a rough local volatility model).
	It consists of a two-dimensional equation, given by
	$$
		dS_t=S_t\sqrt{V_t}\, dW_t,
	$$
	\begin{equation} \label{eq:RoughVol}
		V_t=V_0+\int_0^tK(t-s)\, (\theta-\lambda V_s)~ds+\int_0^tK(t-s)\, \nu \sqrt{V_s}~dB_s,
	\end{equation}
	where W and B are two correlated Brownian motions, with constant correlation $\rho \in (-1,1)$, $\theta,~\lambda ,~\nu $ are positive constants, 
	and the kernel is given by $K(t-s)=(t-s)^{H-\frac{1}{2}}/\Gamma(H+\frac{1}{2})$.
	As the square root function $\sqrt{x}$ is defined only for $x \ge 0$,
	we shall replace $\sqrt{V_s}$ in \eqref{eq:RoughVol} by $\sqrt{ \max(0, V_s)}$ in the numerical implementation to avoid the technical problems when $V_s$ becomes negative in the simulation.
	Unfortunately, our main results require the Lipschitz condition on $\sigma$ (see Assumption \ref{assum:main}), 
	and hence they do not really apply to this case.
	We would like to leave this question for future research.
	We also refer to \cite{AbiJaberElEuch} for an approximation result on this model, where one approximate the kernel $K$ by a sequence of Markovian kernels $K_n$.

	\vspace{0.5em}

	We choose the following parameters: $\lambda=0.3,~\nu=0.3, H=0.1,~V_0=0.02,~\theta=0.02,~ \rho=-0.7$, $S_0 = 1$ and $T=1$.
	We will first estimate the European call option price with strike $1$, that is, $\EE[(S_T-1)_+]$,
	and then consider a path-dependent Asian option with payoff
	$$
		(A_T - 1)_+,
		~~\mbox{where}~
		A_T=\int_0^T S_t~dt.
	$$
	One can compute a reference value of $\EE[(S_T-1)_+]$. 
	As described in \cite{ElEuchRosenbaum}, the characteristic function of the log-price $X_t=\log(S_t/S_0)$ is given by
	$$
		\psi(z)=\EE[e^{izX_t}]
		= 
		\exp(\theta I^1 h(z,t)+V_0 I^{1-\alpha }h(z,t)),
	$$
	where $h(z,.)$ is a solution of the following fractional Ricatti equation
	$$
		D^{\alpha}h(z,t)=\frac{1}{2}(-z^2-iz)+(iz\rho \nu-1)h(z,t)+\frac{( \nu)^2}{2}h^2(z,t),~I^{1-\alpha}h(z,0)=0,
	$$
	with $D^{\alpha}$ and $I^{1-\alpha}$ the fractional derivative and integral defined by
	$$
		D^\alpha f(t)
		=
		\frac{1}{\Gamma(1- \alpha)}\frac{d}{dt}\int_0^t(t-s)^{-\alpha}f(s)~ds
		~~\mbox{and}~
		I^{1-\alpha} f(t)
		=
		\frac{1}{\Gamma(1- \alpha)}\int_0^t(t-s)^{- \alpha}f(s)~ds.		
	$$
	We use the fractional Adams method from \cite{diethelm2004detailed} to solve the above fractional Ricatti equation,
	and hence obtain the function $\psi(\cdot)$.
	Then $\EE[(S_T-1)_+]$ can be obtained by applying the inverse Fourier transform on $\psi(\cdot)$.
	
	\vspace{0.5em}

	However, for the (path-dependent) Asian option pricing $\E[ (A_T - 1)_+ ]$,
	it seems that there is no other method except the Monte Carlo simulation method.
	This is also one of our main motivations to consider the supremum norm error of the discrete time schemes.

	\vspace{0.5em}

	We will implement the Euler scheme and MLMC method.
	For the Euler scheme, we use the uniform discretization, with time step $\Delta t=T/n$ and denote the numerical solution by $(S^n, V^n)$.
	For each simulation, we simulate $N=10^5$ copies $(S^{n,i}, V^{n,i})_{i = 1, \dots N}$ of paths of $(S^n,  V^n)$,
	and then estimate $\E[(S_T - 1)_+]$ and $\E[ (A_T - 1)_+]$ by correspondingly
	$$ 
		\frac1N \sum_{i=1}^N  \big(S^{n,i}_T - 1 \big)_+,
		~~\mbox{and}~
		\frac1N \sum_{i=1}^N \big(A^{n,i}_T - 1 \big)_+,
		~~\mbox{with}~
		A_T^{n,i} ~:=~ \frac{T}{n} \sum_{k=1}^{n} S_{t_k}^{n,i}.
	$$

	For the MLMC method, as the convergence rate result is not available, due to the non-Lipschitz property of the function $\sqrt{x}$,
	we will provide the simulation result of the MLMC method with different level $L$.
	For a given level L, we fix the total simulation number $N=\sum_{\ell=0}^{L}N_\ell=10^5$,
	and then compute the optimal $(N_{\ell}, ~\ell =0, \dots, L)$ to minimize the statistical error under the constraint that $\sum_{\ell = 0}^L N_{\ell} = N$.
	This leads to the optimal allocation $N_\ell\approx \frac{\sqrt{V_\ell}}{\sum_{m=0}^L \sqrt{V_m}}N$, where $V_\ell$ is the (empirical) variance of $(\Ph_\ell- \Ph_{\ell-1})$.
	We choose $M= 4$ so that $h_\ell=4^{-\ell}T$.
	
	\vspace{0.5em}

	The simulation results for European call option and the Asian option are given in Table \ref{tab:3}.
	Again, one can observe the convergence of the Euler scheme as number $n$ of time steps increases, and that of MLMC method as the simulation level $L$ increases.
	For the estimation of the European call option price,
	the relative error of the Euler scheme estimation with $n=160$ is around 2\%, and that of the MLMC estimation with $L=4$ is around 1\%.

\begin{table}[!htbp]
    \centering
    \begin{tabular}{|c||c|c||c|c|}
    \hline
         & Mean(Call) &Stat. Error(Call) & Mean(Asian) & Stat. Error(Asian) 
         \\ \hline

        Reference  & 0.056832&-&-&-
        \\ \hline

        Euler Sch. (n=4) &0.059756 & 0.000245 & 0.040524 & 0.000169
        \\ \hline

        Euler Sch.  (n=10) &0.059138 & 0.000238 & 0.036344 & 0.000145
        \\ \hline

        Euler Sch.  (n=20) &0.058403 & 0.000234 & 0.034551 & 0.000136
        \\ \hline

        Euler Sch.  (n=40) & 0.058494 & 0.000232 & 0.033404 & 0.000131
        \\ \hline

        Euler Sch.  (n=80) & 0.058518 & 0.000232 & 0.033014 & 0.000128
        \\ \hline

         Euler Sch.  (n=160)& 0.058051 & 0.000230 & 0.032626 & 0.000128
	\\ \hline

	MLMC ($L=1$) &0.059875 & 0.000429 & 0.040321 & 0.000435 
	\\ \hline

        MLMC ($L=2$) &0.059249 & 0.000604 & 0.034407 & 0.000548
        \\ \hline

        MLMC ($L=3$) &   0.059014 & 0.000771 & 0.032762 & 0.000643
        \\ \hline

        MLMC ($L=4$) &  0.057497 & 0.000919 & 0.033050 & 0.000733
        \\ \hline
	\end{tabular}
	\caption{European call option and Asian option price estimation in the Rough Heston Model.}
	\label{tab:3}
\end{table}

\section{Proof of Theorems \ref{thm:EulerScheme} and \ref{th:convMilstScheme}}\label{sec:proofs}

	Throughout this section, $C>0$ is a generic constant, whose value may change from line to line.

\subsection{Proof of Theorem \ref{thm:EulerScheme}.(i)}

	The result and proof of Theorem \ref{thm:EulerScheme}.(i) are almost the same to Zhang \cite[Theorem 2.3]{Zhang},
	except that we provide an explicit expression of the convergence rate. 
	We give the proof for completeness, and more importantly, in order to provide this explicit rate. 
	This will also allow for a better presentation of our more original contributions (i.e. Theorem  \ref{thm:EulerScheme}.(ii) and Theorem \ref{th:convMilstScheme}) on the subject.
	
	\vspace{0.5em}

	Let us first repeat and adapt \cite[Lemmas 2.1 and 2.2]{Zhang}, by adding an explicit rate estimation.

	\begin{proposition}\label{prop:Euler}
		Let $p \geq \max( \frac{\beta_1}{\beta_1 -1}, \frac{2\beta_2}{\beta_2 -1})$.
		There exists a constant $C_p \in (0, \infty)$ depending only on $T$, $d$, $p$, and $\beta_1$, $\beta_2$, $C$ in Assumption  \ref{assum:main} such that,
		for all $s, t \in [0,T]$ and $n\ge 1$,
		$$
			\EE \Big[ \big|X_t \big|^p \Big]
			+
			\EE \Big[ \big|X^n_t \big|^p \Big]
			\leq C_p \big(1+\EE \big[|X_0|^p \big] \big),
		$$
		and
		$$
			\EE \Big[ \big|X_t-X_s \big|^p \Big]
			+
			\EE \Big[ \big|X^n_t-X^n_s \big|^p \Big]
			\leq C_p \big( 1+\EE \big[|X_0|^p \big] \big) | t - s |^{p(\alpha\wedge 1)}.
		$$
	\end{proposition}
	\proof
	 Without loss of generality, let us assume that $\EE \big[ |X_0|^p \big]  < \infty$. Under more general conditions on $K_{1}$, $K_{2}$, $b$ and $\sigma$, the existence and uniqueness of $X$ in $L^p \left([0,T]\times \Omega\right)$ is established in the proof of Theorem 1.1 of \cite{Wang}, using a fixed point argument.
	
	\vspace{0.5em}
	
	\noindent $(i)$ Let us consider first the estimation of $\E[ |X_t|^p]$.
	
	\vspace{0.5em}
	
	Using \eqref{eq:defX},
	it follows by the H\"older inequality and the BDG inequality that
	\begin{eqnarray*}
		\EE \big[ |X_t|^p \big]
		&\leq&
		C \EE \big[ |X_0|^p \big]
		+
		C \EE \Big[ \Big| \int_{0}^{t}   K_{1}(t,s)b(s, X_s) ~ds\Big|^{p} \Big]\\
		&&
		+~
		C \EE \Big[ \Big| \int_{0}^{t}K_{2}(t,s)\sigma(s,X_s)~dW_{s}\Big|^{p} \Big]
		\\
		&\leq&
		C \EE \big[ |X_0|^p \big]
		+
		C \EE \Big[
			\Big( \int_0^{t} \big| K_1(t,s) \big|^{\beta_1} ds \Big)^{\frac{p}{\beta_1}}
			\Big(\int_0^{t} \big| b(s,X_s) \big|^{\frac{\beta_1}{(\beta_1-1)}} ds \Big)^{\frac{p(\beta_1-1)}{\beta_1}}
		\Big]
		\\
		&&+~
		C \EE \Big[ \Big( \int_0^{t }  \big |K_2(t,s) ~\sigma (s,X_s) \big|^2  ds \Big)^{\frac{p}{2}} \Big].
	\end{eqnarray*}
	Further, notice that $\sup_{t \in [0,T]} \big( |b(t,0)| + |\sigma(t,0)| \big)$ is uniformly bounded by the H\"older  continuity of $\big( b(t,0), \sigma(t,0) \big)$ in $t$ (Condition \ref{eq:B}).
	Using again the Lipschitz condition in Condition \ref{eq:B},
	one has $| b(s,x) | + | \sigma(s,x) |\leq C(1+|x|)$ for some constant $C > 0$.
	Then by using \ref{eq:A1} and  H\"older's inequality \big(recall that $\beta_1>1$, $\beta_2 > 1$ and $p \ge \max( \beta_1/(\beta_1-1), 2\beta_2 /(\beta_2 -1))$\big),
	one obtains a constant $C$ such that
	\begin{align*}
		\EE \big[ |X_t|^p \big]
		&\leq
		C \EE \big[ |X_0|^p\big]
		+
		C \EE \Big[ \int_0^{t} \big(1+|X_s|^p \big)~ds \Big]
		+
		C \EE \Big[ \Big(\int_0^{t} \big| \sigma(s,X_s) \big|^{\frac{2\beta_2}{\beta_2-1}}~ds  \Big)^{\frac{p (\beta_2 -1)}{2\beta_2}} \Big]
		\\
		&\leq
		C \big( 1+ \EE \big[|X_0|^p \big]  \big)
		+
		C \int_{0}^{t} \EE \big[ |X_{s}|^{p} \big] ds~.
	\end{align*}
	The result then follows by Grönwall's lemma.

~

	\noindent $(ii)$
	We  consider next the estimation of $\E \big[ |X^n_t|^p \big]$, where the proof is almost the same.
	Indeed, we have to consider here the integrals
	\begin{equation} \label{eq:seq_K_etan}
		\int_0^{t} |K_1(t,\ens)|^{\beta_1}~ds
		\quad \mbox{ and } \quad
		\int_0^{t} |K_2(t,\ens)|^{2\beta_2}~ds .
	\end{equation}
	These are Riemann sums which therefore converge, as $n\to \infty$, respectively to $\int_0^{t} |K_1(t,s)|^{\beta_1}~ds$ and  $\int_0^{t} |K_2(t,s)|^{2\beta_2}~ds$,
	which are finite real values by \ref{eq:A1}.
	As any convergent sequence of real numbers is uniformly bounded, the two sequences in \eqref{eq:seq_K_etan} are uniformly bounded in $n$.
	Then one can conclude as in $(i)$ that, for some constant $C$ independent of $n \ge 1$ and $t \in [0,T]$,
	$$
		\EE \big[ |X_t^n|^p \big]  \leq C \big (1+\EE \big[ |X_0|^p \big] \big).
	$$

	\noindent  $(iii)$ Let $s < t$ and denote $\delta := t-s$, we consider the term $ \E \big[ |X_t - X_s |^p \big]$.
	Let us rewrite
	\begin{eqnarray*}
		X_t-X_s
		&=&
		\int_{s}^t K_1 (t,u)b(u,X_u)~du+\int_{s}^t K_2(t,u)\sigma(u,X_u)\, dW_u \\
		&&
		+ \int_0^{s}\Big(K_1(t,u)b(u,X_u) -K_1 (s,u)b(u,X_u)\Big)\, du\\
		&&
		+ \int_0^{s}\Big(K_2(t,u)\sigma(u,X_u)-K_2(s,u)\sigma(u,X_u)\Big)\, dW_u \\
		&=:&
		I_1+I_2+I_3+I_4,
	\end{eqnarray*}
	and then consider  $I_1$, $I_2$, $I_3$, $I_4$ separately.
	
	\vspace{0.5em}
	
	\noindent For $I_1$, by applying Minkowski's integral inequality (see \cite[p.271]{Stein}) and condition \ref{eq:A2},
	it follows that
	\begin{align*}
		\EE \big[ |I_1|^p \big]
		&~\leq~
		\Big( \int_{s}^t |K_1(t,u)| ~\big( \E [ |b(u,X_u)|^p] \big)^{1/p} du \Big)^p
		~\leq~
		C  \big( 1 + \EE [ |X_0|^p ] \big) ~ \delta^{p (\alpha_1 \wedge 1)}.
	\end{align*}
	For $I_2$, we apply BDG's inequality, Minkowski's integral inequality, Conditions \ref{eq:A2}, and \ref{eq:B} on $\sigma$,
	it follows that
	\begin{eqnarray*}
		\EE \big[ |I_{2}|^{p} \big]
		&\leq&
		C\EE \Big[ \Big(\int_{s}^{t} | K_{2}(t,u)|^{2} ~ |\sigma(u,X_{u})|^{2}~du\Big)^{\frac{p}{2}} \Big]
		\leq
		C \bigg(\int_{s}^t\Big(\EE \Big[ |K_{2}(t,u)|^{p} |\sigma(u,X_{u})|^{p}  \Big]  \Big)^{\frac{2}{p}} ~du\bigg)^{\frac{p}{2}} \\
		&\leq&
		C \big(1+\EE \big[|X_0|^p \big]\big) \delta^{p (\alpha_2 \wedge 1)}.
	\end{eqnarray*}
	For $I_3$, we use Minkowski's integral inequality and Condition \ref{eq:A3} to obtain that
	\begin{align*}
		\EE|I_{3}|^{p}
		&~\leq~
		\Big(  \int_0^s  \big| K_1(t,u)-K_1(s,u)\big|  \big( \E \big[ | b(u,X_u)|^p \big] \big)^{1/p} du \Big)^p
		~\leq~
		C\big(1+\EE \big[|X_0|^p \big]\big) \delta^{p(\alpha_1 \wedge 1)}~.
	\end{align*}
	For $I_4$, we apply BDG's inequality, Minkowski's integral inequality  and use \ref{eq:A3} to obtain that
	\begin{align*}
		\EE|I_4|^p
		&~\leq~
		C\EE\Big(\int_0^{s}\big|K_2(t,u)-K_2(s,u)\big|^2 |\sigma(u,X_u)|^2~du\Big)^{\frac{p}{2}} \\
		&~\leq~
		C\bigg(\int_0^{s} \Big(\EE| K_2(t,u)-K_2(s,u) |^p|\sigma(u,X_u)|^p\Big)^{\frac{2}{p}}~du\bigg)^{\frac{p}{2}}
		~\leq~ C\big(1+\EE \big[|X_0|^p \big]\big) \delta^{p(\alpha_2\wedge 1)}.
	\end{align*}
	Then it follows that
	$$
		\EE \big[ |X_t-X_s|^p \big]
		~\leq~
		C \EE \big[ I_1|^p+| I_2|^p+| I_3|^p+| I_4|^p \big]
		~\leq~
		C \big(1+\EE \big[|X_0|^p \big]\big) (t-s)^{p (\alpha\wedge 1)}.
	$$

	\noindent $(iv)$ Finally, for the estimation of $\EE \big[ |X^n_t-X^n_s|^p \big]$, one can similarly write
	\begin{align*}
		X^n_t-X^n_{s} =&\int_{s}^t K_1 (t,\enu)b(\enu,X_{\enu})~du+\int_{s}^t K_2(t,\enu)\sigma(\enu,X_{\enu})~dW_u\\
		&+ \int_0^{s}\Big(K_1(t,\enu)b(\enu,X_{\enu}) -K_1 (s,\enu)b(\enu,X_{\enu})\Big)~du\\
		&+ \int_0^{s}\Big(K_2(t,\enu)\sigma(\enu,X_{\enu})-K_2(s,\enu)\sigma(\enu,X_{\enu})\Big)~dW_u\\
		&=:I^n_1 + I^n_2 + I^n_3 + I^n_4.
	\end{align*}
	Notice that the conditions in \ref{eq:A2} and \ref{eq:A3} are given also on 
	$$
		\int_t^{t'} \|K_i(t',\ens)\|^i~ds 
		~~\mbox{and}~~
		\int_0^{t} \|K_i(t+\delta,\ens)-K_i(t,\ens)\|^i~ds,
		~~i=1,2,
	$$
	one can apply the same arguments to obtain the estimations for
	$\E \big[ | I^n_1 |^p \big], \dots, \E \big[ | I^n_4 |^p \big]$.
	\qed

	\vspace{0.5em}

\noindent {\it Proof of Theorem \ref{thm:EulerScheme}.$(i)$}.
	Let us rewrite
	\begin{align*}
		X_t-X_t^n =&\int_0^t\Big(K_1(t,s)b(s,X_s)-K_1(t,\ens)b(\ens,X_{\ens}^n)\Big)~ds\\
		&+\int_0^t\Big(K_2(t,s)\sigma(s,X_s)-K_2(t,\ens)\sigma(\ens,X_{\ens}^n)\Big)~dW_{s} \\
		=&\int_0^t\Big(K_1(t,s)-K_1(t,\ens)\Big)b(s,X_s)~ds+\int_0^t
		K_1(t,\ens)\Big(b(s,X_s)-b(\ens,X_{\ens})\Big)~ds\\
		&+\int_0^t K_1(t,\ens)\Big(b(\ens,X_{\ens})-b(\ens,X_{\ens}^n)\Big)~ds\\
		&+\! \int_0^t \!\! \big(K_2(t,s)-K_2(t,\ens)\big)\sigma(s,X_s)dW_{s}
			+\! \int_0^t \!\! K_2(t,\ens)\big(\sigma(s,X_s)-\sigma(\ens,X_{\ens})\big)dW_s \\
		&+\int_0^tK_2(t,\ens)\Big(\sigma(\ens,X_{\ens})-\sigma(\ens,X_{\ens}^n)\Big)~dW_s\\
		=&: J_1+J_2+J_3+J_4+J_5+J_6,
	\end{align*}
	and then consider $J_1, \dots, J_6$ separately.
	
	\vspace{0.5em}

	\noindent For $J_1$, we use Minkowski's integral inequality, Proposition \ref{prop:Euler} and \ref{eq:A4} to obtain that
	\begin{align*}
		\EE \big[ |J_1|^p  \big]
		~&\leq~
		\Big( \int_0^t \big|K_1(t,s)-K_1(t,\ens)\big|~ \big( \E [|b(s,X_s)|^p] \big)^{1/p} ~ds\Big)^p
		~\leq~
		C \big(1+\EE \big[|X_0|^p \big]\big) ~\dn^{p(\alpha_1 \wedge 1)} .
	\end{align*}
	For $J_2$, notice that $\int_0^t |K_1(t,\ens)|~ds < \infty$  by \ref{eq:A2},
	then  it follows by  Minkowski's integral inequality together with Condition \ref{eq:B} and Proposition \ref{prop:Euler} that
	\begin{align*}
		\EE \big[ |J_2|^p \big]
		&\leq
			C\EE \Big[ \Big(\int_0^t\big|K_{1}(t,\ens)\big(b(s,X_s)-b(\ens,X_s)\big)\big|~ds\Big)^p \Big] \\
		&\quad +
			C\EE \Big[ \Big(\int_0^t\big|K_1(t,\ens)\big(b(\ens,X_{s})-b(\ens,X_{\ens}))\big|~ds\Big)^{p} \Big] \\
		&\leq
			C \Big( \int_0^t \big| K_1(t, \ens) \big| ~(s -\ens)^{\alpha\wedge 1} ~\big( \E [ ( 1 + |X_s|^p) ] \big)^{1/p} ~ds\Big)^p 
		\\
		&\quad +
			C \Big( \int_0^t \big| K_1(t, \ens) \big|~ \big( \E \big[ |X_s - X_{\ens}|^p \big] \big)^{1/p} ~ds\Big)^p 
		\\
		&\leq
			C \big(1+\EE \big[|X_0|^p \big]\big) ~\dn^{p(\alpha \wedge 1)}.
	\end{align*}
	For $J_3$, we obtain by H\"{o}lder's inequality and $\int_0^t |K_1(t,\ens)|^{\beta_1}~ds <\infty$ that
	\begin{align*}
		\EE \big[ |J_3|^p \big]
		&\leq
		C \EE \Big[ \Big( \int_0^t \big|X_{\ens}-X_{\ens}^n\big|^{\beta_1/(\beta_1-1)}~ds \Big)^{\frac{p (\beta_1 -1)}{\beta_1}} \Big]
 		~\leq~
		C\int_0^t\EE \big[ \big|X_{\ens}-X_{\ens}^n\big|^p \big]~ds.
	\end{align*}
	For $J_4$, it follows  by BDG's inequality, Minkowski's integral inequality and \ref{eq:A4} that
	\begin{align*}
		\EE \big[ |J_4|^p \big]
		&\leq
		C\EE \Big[ \Big( \int_0^t\Big|K_2(t,s)-K_2(t,\ens)\Big|^2|\sigma(s,X_s)|^2~ds \Big)^{\frac{p}{2}} \Big] \\
		& \leq
		C \Big(\int_0^t\Big(\EE\Big[ |\sigma(s,X_s)|^p\big|K_2(t,s)-K_2(t,\ens)\big|^p\Big] \Big)^{\frac{2}{p}}~ds \Big)^{\frac{p}{2}}
		~\leq~
		C \big(1+\EE \big[|X_0|^p \big]\big) \delta^{p(\alpha_2 \wedge 1)}.
	\end{align*}
	For $J_5$, we use BDG's inequality, Minkowski's integral inequality,
	Proposition \ref{prop:Euler}  and the H\"older regularity in time of $\sigma$ (Condition \ref{eq:B}) to obtain that
	\begin{align*}
		\EE \big[ |J_5|^p \big]
		&~\leq~
		C \EE \Big[ \Big( \int_0^t |K_2(t,\ens)|^2 \Big| \sigma(s,X_s)-\sigma(\ens,X_{\ens})\Big|^2~ds\Big)^{\frac{p}{2}} \Big]
		~\leq~
		C \big(1+\EE \big[|X_0|^p \big]\big)  \dn^{ p(\alpha\wedge 1)} .
	\end{align*}
	For $J_6$, we have by BDG's inequality and H\"{o}lder's inequality that, for $\beta_2>1$ that appears in \ref{eq:A1},
	\begin{align*}
		\EE \big[ |J_6|^p \big]
		&\leq
		\EE \bigg[ \bigg(\int_0^t |K_2(t,\ens)|^2\Big|\sigma(\ens,X_{\ens})-\sigma(\ens,X_{\ens}^n)\Big|^2~ds\bigg)^{\frac{p}{2}} \bigg] \\
		&\leq
		\EE \bigg[ \left(\int_0^t |K_2(t,\ens)|^{2\beta_2} ~ds\right)^{\frac{p}{2\beta_2}}
			\left( \int_{0}^t \Big|\sigma(\ens,X_{\ens})-\sigma(\ens,X_{\ens}^n)\Big|^{\frac{2\beta_2}{\beta_2-1}}~ds\right)^{\frac{p(\beta_2-1)}{2\beta_2}} \bigg] \\
		& \leq C\int_{0}^{t}\EE|X_{\ens}-X_{\ens}^{n}|^{p}~ds,
	\end{align*}
	where in the last lign we used again H\"older's inequality with $p\geq \frac{2\beta_2}{\beta_2-1}$.
	
	\vspace{0.5em}
	
	Combining all the above estimations, it follows that
	$$
		\EE \big[ |X_{t}-X_t^n|^p \big]
		~\leq~
		C \big(1+\EE \big[|X_0|^p \big]\big) \dn^{p (\alpha \wedge 1)}+C\int_0^t \sup\limits_{u\in [0,s]}\EE \big[ |X_u-X_u^n|^p \big] ds.$$
	Then by Grönwall's Lemma, we conclude that
	$\sup\limits_{t\in [0,T]}\EE \big[ |X_{t}-X_{t}^{n}|^{p} \big]  \leq C \big(1+\EE \big[|X_0|^p \big]\big) \dn^{p (\alpha \wedge 1)} $ for some constant $C > 0$ independent of $n$ and $X_0$.
	\qed

\subsection{Proof of Theorem \ref{th:convMilstScheme}.$(i)$}

	We now consider the solution $\Xb^n$ to the Milstein scheme \eqref{eq:MilsteinScheme}.
	 For ease of presentation, we consider the one-dimensional case with $d=1$,
	and write $b'$ (resp. $\sigma'$) in place of $\nabla_x b$ (resp. $\nabla_x \sigma$).
	The high-dimensional case will only change the generic constant $C$ depending on $d$. 
	Similarly to Proposition \ref{prop:Euler}, we first provide some related \emph{a priori} estimations.

	\begin{proposition}\label{prop:Milstein}
		Let Assumptions \ref{assum:main} and \ref{assum:main2} hold true,
		and $p \geq \max( \frac{\beta_1}{\beta_1 -1}, \frac{2\beta_2}{\beta_2 -1})$.
		Then there exists a constant $C_p \in(0,\infty)$ depending only on $T$, $d$, $p$ and $\beta_1$, $\beta_2$, $C$ in Assumptions \ref{assum:main} and \ref{assum:main2} such that, for all $s, t \in [0,T]$ and $n \ge 1$,
		\begin{equation} \label{eq:apriori_Milst}
			\EE \Big[ \big| \Xb_{t}^n \big|^p \Big] \leq C_p \big(1+\EE \big[|X_0|^p \big]\big),
			~~~
			\EE \Big[ \big| \Xb_t^n-\Xb_s^n \big|^p \Big] \leq C_p \big(1+\EE \big[|X_0|^p \big]\big) |t - s|^{p (\alpha \wedge 1)},
		\end{equation}
		and
		\begin{equation} \label{eq:apriori_Milst2}
			\EE \Big[ \big|A_s^{1,n} \big|^p\Big] + \EE \Big[ \big|A_s^{2,n} \big|^p\Big]
			\leq 
			C_p \big(1+\EE \big[|X_0|^p \big]\big) ~ \dn^{p(\alpha \wedge 1)} .
		\end{equation}
	\end{proposition}
	\begin{proof}
	 $(i)$ Let us first consider the term $\EE \big[ \big| \Xb_{t}^n \big|^p \big]$.
	Notice that the solution $\Xb^n$ is essentially defined on the discrete-time grid $\{ t^n_k, ~k=0, \dots, n\}$.
	When $\E \big[ |X_0|^p \big] < \infty$, using the induction argument and Condition \ref{eq:B}, together with the boundedness of $b'$ and $\sigma'$,
	it is easy to deduce  that
	$\E \big[ \big| \Xb^n_t \big|^p \big] < \infty$ for every $n \ge 1$ and $t \in [0,T]$.
	Then we do not really need to localise the process $\Xb^n$ to obtain the \emph{a priori} estimation.

	\vspace{0.5em}

	First, by Condition \ref{eq:A1}, one has
	$$
		\int_0^{\ens}|K_2(s,r)-K_2(\ens,r)|^{2\beta_2}~dr\leq C\left(\int_0^{\ens} |K_2(s,r)|^{2\beta_2}+|K_2(\ens,r)|^{2\beta_2}~dr\right)< \infty.
	$$
	Then by the BDG inequality and H\"{o}lder's inequality with $\beta_2>1$, it follows that
	\begin{align}\label{eq:boundAn}
		&
		\EE \Big[ \big|A_s^{1,n} \big|^p\Big] + \EE \Big[ \big|A_s^{2,n} \big|^p\Big] \nonumber \\
		&\leq
		C \EE \Big[ \Big( \int_0^{\ens}\big|K_2(s,r)-K_2(\ens,r)\big|^2~  \big| \sigma(\eta_{n}(r),\Xb_{\eta_{n}(r)}^n) \big|^2~dr\Big)^{\frac{p}{2}} \Big] \nonumber  \\
		&~~~+
		C \EE \Big[ \Big( \int_{\ens}^{s} |K_2(s,r) |^2  \big| \sigma(\eta_{n}(r),\Xb_{\eta_{n}(r)}^n) \big|^2~dr\Big)^{\frac{p}{2}} \Big]  \nonumber \\
		&\leq C \EE \bigg[
			 \bigg( \int_0^{\ens}\big|K_2(s,r)-K_2(\ens,r)\big|^{2\beta_2}~dr \bigg)^{\frac{p}{2\beta_2}}
			 \bigg( \int_{0}^{\ens}|\sigma(\eta_{n}(r),\Xb_{\eta_{n}(r)}^n)|^{\frac{2\beta_2}{\beta_2-1}}~dr \bigg)^{\frac{p(\beta_2-1)}{2\beta_2}}
		\bigg] \nonumber \\
		&\quad\quad +
		C \EE\bigg [
			\bigg(\int_{\ens}^{s} |K_2(s,r)|^{2\beta_2} ~dr\bigg)^{\frac{p}{2\beta_2}}
			\bigg(\int_{\ens}^{s} \sigma(\eta_{n}(r),\Xb_{\eta_{n}(r)}^n)|^{\frac{2\beta_2}{\beta_2-1}}~dr\bigg)^{\frac{p(\beta_2-1)}{2\beta_2}}
		\bigg]  \nonumber\\
		&\leq
		C \EE \Big[ \int_0^s \big| \sigma \big(\eta_{n}(r),\Xb_{\eta_{n}(r)}^n \big) \big|^p ~dr \Big]
		~\leq~
		C \int_0^s \Big(1+\EE \big[ \big|\Xb_{\eta_{n}(r)}^n \big|^p \big] \Big)~dr ,
	\end{align}
	where we applied H\"older's inequality for the second inequality with $\frac{p(\beta_2-1)}{2\beta_2}\geq 1$. 
	Next, applying again the BDG inequality and then H\"older's inequality as before,
	\begin{align*}
		\EE \big[ \big| \Xb_{t}^n \big|^p \big]
		&\leq
		C \EE\bigg[ \big|X_0|^p + \Big|\int_0^t K_1(t,s)\left(b(\ens,\Xb_{\ens}^n)+b'(\ens, \Xb_{\ens}^n) A_{s}^{1, n}\right)~ds \Big|^p \\
		&~~~~~~~~~~
			+\Big(\int_0^t |K_2(t,s)|^2\big|\sigma(\ens, \Xb_{\ens}^n)+\sigma'(\ens,\Xb_{\ens}^n) A_s^{n}\big|^2~ds\Big)^{\frac{p}{2}}\bigg] \\
		&\leq
		C \EE \big[ | X_0|^p \big]
		+
		C\int_0^{t } \E \Big[ \Big|b(\ens,\Xb_{\ens}^n)+b'(\ens,\Xb_{\ens}^n)A_{s}^{1, n}\Big|^p \Big]~ds\\
		&~~~~~~~~~
		+
		C\int_0^{t} \E \Big[ \Big|\sigma(\ens,\Xb_{\ens}^n)+\sigma'(\ens,\Xb_{\ens}^n) A_s^{n}\Big|^p \Big]~ds.
	\end{align*}
	By the boundedness condition of $b'$ and $\sigma'$ in Assumption \ref{assum:main2},
	it follows that
	\begin{eqnarray*}
		\EE \big[ \big|\Xb_{t}^n \big|^p \big]
		&\leq&
		C \Big(\EE \big[ |X_0|^p \big] +1+ \int_0^{t}\EE \big[ \big| \Xb_{\ens}^n \big|^p \big] ds
		+\int_0^{t}  \big(  \EE\big[ \big|A_s^{1, n}\big|^p \big] +  \EE\big[ \big|A_s^{2, n}\big|^p \big] \big)  ds \Big) \\
		&\leq&
		C \big(1+\EE \big[|X_0|^p \big]\big) + C \int_0^{t} \sup\limits_{u \in [0,s]}\EE\big[ \big |\Xb_{u}^n \big|^p \big] ds .
	\end{eqnarray*}
	Then we obtain the first estimation in \eqref{eq:apriori_Milst} by Grönwall's Lemma.
	
	\vspace{0.5em}
	
	\noindent $(ii)$ Let $s<t$. By direct computation, we write
	\begin{align*}
		\Xb_t^n-\Xb_{s}^n&=\int_0^{s} \Big(K_1(t,u)-K_1(s,u)\Big)\Big( b(\enu,\Xb_{\enu}^n)+b'(\enu,\Xb_{\enu}^n)  A_{u}^{1,n} \Big)~du\\
		&~~~+\int_{s}^t K_1(t,u)\Big(b(\enu,\Xb_{\enu}^n)+b'(\enu,\Xb_{\enu}^n)  A_{u}^{1,n}\Big)~du\\
		&~~~+\int_0^{s} \Big(K_2(t,u)-K_2(s,u)\Big)\Big( \sigma(\enu,\Xb_{\enu}^n)+\sigma'(\enu,\Xb_{\enu}^n) A_u^{n}\Big)~dW_u\\
		&~~~+\int_{s}^t K_2(t,u)\Big(\sigma(\enu,\Xb_{\enu}^n)+\sigma'(\enu,\Xb_{\enu}^n) A_u^{n}\Big)~dW_u\\
		&=:I_1+I_2+I_3+I_4.
	\end{align*}
	For $I_3$, we deduce from the BDG inequality and Minkowski's integral inequality that
	\begin{align*}
		\EE \big[ |I_3|^p \big]
		& \leq
		C
		\EE \Big[ \Big(\int_0^{s}\Big|K_2(t,u)-K_2(s,u)\Big|^2\Big|\sigma(\enu,\Xb_{\enu}^n)+\sigma'(\enu,\Xb_{\enu}^n) A_u^{n}\Big|^2 ~du \Big)^{\frac{p}{2}} \Big]\\
		&\leq
		C
		\bigg(\int_0^{s}\bigg(  \EE \bigg[\big|K_2(t,u)-K_2(s,u)\big|^{p} ~\Big|\sigma(\enu,\Xb_{\enu}^n)+\sigma'(\enu,\Xb_{\enu}^n)  A_u^{n}\Big|^p\bigg]
\bigg)^{\frac{2}{p}}~du \bigg)^{\frac{p}{2}} .
	\end{align*}
	Notice that $\EE| A_u^{n} |^p \leq C(1+\EE |X_{0}|^p)$ by \eqref{eq:boundAn} and the first estimation in \eqref{eq:apriori_Milst},
	hence it follows by Condition \ref{eq:A3} that
	\begin{align*}
		\EE \big[ |I_3|^p \big]
		~\leq~
		C \big(1+\EE \big[|X_0|^p \big]\big) \bigg(\int_0^{s} \Big|K_2(t,u)-K_2(s,u)\Big|^2~du \bigg)^{\frac{p}{2}}
		~\leq~
		C  \big(1+\EE \big[|X_0|^p \big]\big) (t-s)^{p(\alpha_2\wedge 1)}.
	\end{align*}
	For $I_4$, we use BDG's inequality, Minkowski's integral inequality,  the first estimation in \eqref{eq:apriori_Milst} and then \ref{eq:A2}  to deduce that
	\begin{align*}
		\EE \big[ |I_4|^p \big]
		&\leq C\EE \Big[ \Big( \int_{s}^t |K_2(t,u)|^2\big|\sigma(\enu,\Xb_{\enu}^n)+\sigma'(\enu,\Xb_{\enu}^n) A_u^{n}\big|^2~du \Big)^{\frac{p}{2}} \Big] \\
		&\leq C \left( \int_{s}^t \Big(\EE \Big[|K_2(t,u)|^p \big|\sigma(\enu,\Xb_{\enu}^n)+\sigma'(\enu,\Xb_{\enu}^n) A_u^{n}\big|^p\Big]\Big)^{\frac{2}{p}} ~du\right)^{\frac{p}{2}}\\
		&\leq C \big(1+\EE \big[|X_0|^p \big]\big) \left(\int_{s}^t |K_2(t,u)|^2~du\right)^{\frac{p}{2}}
		~\leq~ C \big(1+\EE \big[|X_0|^p \big]\big) (t-s)^{p(\alpha_2 \wedge 1)}.
	\end{align*}
	Further, by similar arguments, one can also obtain the estimation on $I_1$ and $I_2$:
	$$
		\EE \big[ |I_1|^p \big]
		+
		\EE \big[ |I_2|^p \big]
		~\leq~
		C \big(1+\EE \big[|X_0|^p \big]\big) (t-s)^{p(\alpha_1 \wedge 1)},
	$$
	and it follows that
	$$
		\EE \big[ \big| \Xb_t^n-\Xb_{\ent}^n \big|^p \big]
		~ \leq~
		C \big(\EE|I_1|^p+\EE|I_2|^p+\EE|I_3|^p+\EE|I_4|^p\big)
		\leq
		C \big(1+\EE \big[|X_0|^p \big]\big) \dn^{p(\alpha \wedge 1)} .
	$$

	\noindent $(iii)$ Finally, using the first estimation in \eqref{eq:apriori_Milst},
	one obtains from the BDG inequality and Minkowski's integral inequality that
	\begin{align*}
		 \EE \big[ |A_s^{1,n}|^p \big]
		+
		\EE \big[ |A_s^{2, n}|^p \big]
		&\leq
		C \left( \int_0^{\ens} \left(\EE\left[\left| K_2(s,r)-K_2(\ens,r)\right|^p~  |\sigma(\eta_{n}(r),\Xb_{\eta_{n}(r)}^n)|^p\right]\right)^{\frac{2}{p}}~dr\right)^{\frac{p}{2}}\\
		&~  +C \left(  \int_{\ens}^{s}   \left(\EE\left[ | K_2(s,r)|^p |\sigma(\eta_{n}(r),\Xb_{\eta_{n}(r)}^n)|^p\right]\right)^{\frac{2}{p}}~dr\right)^{\frac{p}{2}}
 		\\
		&~\le~
		C \big(1+\EE \big[|X_0|^p \big]\big)~ \dn^{p(\alpha_2\wedge 1)}
		~\le~
		C \big(1+\EE \big[|X_0|^p \big]\big)~ \dn^{p(\alpha \wedge 1)}.
	\end{align*}
	\end{proof}

	\noindent {\it Proof of Theorem \ref{th:convMilstScheme}.$(i)$.}
	Let us rewrite
	\begin{align*}
		X_t-\Xb_t^n =& \int_0^t K_1(t,s)\bigg(b(s,X_s)-\Big( b(\ens,\Xb_{\ens}^n)+b'(\ens,\Xb_{\ens}^n)   A_{s}^{1,n} \Big) \bigg)~ds \\
		&+\int_0^t K_2(t,s)\bigg(\sigma(s,X_s)- \Big( \sigma(\ens,\Xb_{\ens}^n)+\sigma'(\ens,\Xb_{\ens}^n)  A_s^{n}\Big)\bigg)~dW_s\\
		=& \int_0^t K_1(t,s)\Big( b(s,X_s)-b(s,\Xb_s^n)\Big)~ds +\int_0^t K_1(t,s)\Big( b(s,\Xb_s^n)-b(\ens,\Xb_s^n)\Big)~ds\\
		&+\int_0^t K_1(t,s)\bigg(b(\ens,\Xb_s^n)-\Big(b(\ens,\Xb_{\ens}^n)+b'(\ens,\Xb_{\ens}^n) A_s^{n} \Big) \bigg)~ds\\
		&+ \int_0^t K_1(t,s)b' \big(\ens,\Xb_{\ens}^n \big)  A_s^{2,n} ~ds \\
		&+\int_0^t K_2(t,s)\Big( \sigma(s,X_s)-\sigma(s,\Xb_s^n)\Big)~dW_s +\int_0^t K_2(t,s)\Big( \sigma(s,\Xb_s^n)-\sigma(\ens,\Xb_s^n)\Big)~dW_s\\
		&+\int_0^t K_2(t,s)\Big(\sigma(\ens,\Xb_s^n)-\Big(\sigma(\ens,\Xb_{\ens}^n)+\sigma'(\ens,\Xb_{\ens}^n) A_s^{n} \Big) \Big)~dW_s\\
		=&: ~J_1+ J_2+ J_3+ J_4+ J_5+ J_6+J_7 ~.
	\end{align*}
	For $J_1$, by similar computations as in Theorem \ref{thm:EulerScheme}, it is easy to obtain that
	$$
		\EE \big[ |J_1|^p \big]
		\leq C \int_0^t \EE \Big[ \big|X_s-\Xb_s^n \big|^p \Big]~ds~.
	$$
	For $J_2$, notice that $\int_0^t |K_1(t,s) |^{\beta_1}~ds <\infty$ and $p \ge \beta_1/(\beta_1-1)$,
	we have by H\"{o}lder's inequality,  Condition \ref{eq:Bt} and Proposition \ref{prop:Milstein} that
	\begin{align*}
		\EE \big[ |J_2|^p \big]
		~&\leq ~
		C\EE \Big[ \int_0^t (s-\ens)^{p(2\alpha' \wedge 1)} \big( 1+|\Xb_s^n|^p \big)~ds \Big]
		~\leq~
		C\big( 1+\EE \big[|X_0|^p \big] \big) \dn^{p(2\alpha' \wedge 1)} .
	\end{align*}
	For $J_3$, a Taylor expansion gives
	$$
		b(\ens,\Xb_s^n)
		~=~
		b(\ens,\Xb_{\ens}^n)+b'(\ens,\Xb_{\ens}^n)(\Xb_s^n-\Xb_{\ens}^n)+\epsilon_s^n ,
	$$
	where $|\epsilon_s^n| \leq C|\Xb_s^n-\Xb_{\ens}^n|^2$ (using that the second derivative of $b$ is bounded).
	Then by Minkowski's integral inequality, one has
	\begin{align}\label{eq:bound0I3}
		\EE \big[ |J_3|^p \big]
		&\leq
		\left(\int_{0}^t  \left( \EE \Big[ \Big| K_{1}(t,s)\left( b'(\ens, \Xb^n_{\ens})(\Xb^n_{s}-\Xb^n_{\ens}- A_u^{n}) + \epsilon^n_{s}\right) \Big|^p \Big] \right)^{\frac{1}{p}} ~ds\right)^p.
	\end{align}
	Thus using the boundedness of $b'$ (Assumption \ref{eq:Bt}) and the definitions of $\Xb^n$ and $A^n$   in \eqref{eq:MilsteinScheme}-\eqref{eq:defAn},
	\begin{align*}
		\EE& \left[ \left| b'(\ens, \Xb^n_{\ens})(\Xb^n_{s}-\Xb^n_{\ens}- A_s^{n}) + \epsilon^n_{s} \right|^p \right] \\
		& \hspace{1cm} \leq C \bigg\{ \EE  \Big[ \Big|\int_0^{\ens} \Big(K_1(s,r)-K_1(\ens,r) \Big)\Big(b(\enr, \Xb_{\eta_{n}(r)}^n)+b'(\enr,\Xb_{\eta_{n}(r)}^n) A_{r}^{1,n} \Big)~dr\Big|^p \Big]\\
		&\hspace{1.6cm} +\EE \Big[ \Big|\int_0^{\ens} \left( K_2(s,r)-K_2(\ens,r) \right) ~ \sigma '(\enr,\Xb_{\enr}^n) A_r^{n} ~dW_r\Big|^p \Big]\\
		&\hspace{1.6cm} +\EE \Big[ \Big|\int_{\ens}^{s}K_1(s,r)\left(b(\enr,\Xb_{\eta_{n}(r)}^n)+b'(\enr,\Xb_{\eta_{n}(r)}^n) A_{r}^{1,n} \right)~dr\Big|^p \Big]\\
		&\hspace{1.6cm} +\EE \Big[ \Big|\int_{\ens}^s K_2(s,r)  ~ \sigma'(\enr,\Xb_{\eta_{n}(r)}^n) A_r^{n} ~dW_r\Big|^p \Big]
			+ \EE \big[ |  \epsilon^n_{s} |^p \big] \bigg\} .
	\end{align*}
	We apply Minkowski's integral inequality for the first and third summand,
	and the BDG inequality for the second and the fourth,
	in order to obtain
	\begin{align*}
		\EE& \Big[ \Big| b'(\ens, \Xb^n_{\ens})(\Xb^n_{s}-\Xb^n_{\ens}- A_s^{n}) + \epsilon^n_{s} \Big|^p \Big] \\
		& \hspace{0.2cm}
		\leq
		C \bigg\{
		\left(\int_0^{\ens} \left(\EE \Big[ \Big| \left(K_1(s,r)-K_1(\ens,r) \right) \left(b(\enr, \Xb_{\eta_{n}(r)}^n)+b'(\enr,\Xb_{\eta_{n}(r)}^n) A_{r}^{1,n} \right)\Big|^p \Big] \right)^{\frac{1}{p}}~dr\right)^p\\
		&\hspace{1.6cm} +\EE \Big[ \Big(\int_0^{\ens} \left( K_2(s,r)-K_2(\ens,r) \right)^2 ~ \left(\sigma '(\enr,\Xb_{\enr}^n) A_r^{n}\right)^2 ~dr\Big)^{\frac{p}{2}} \Big] \\
		&\hspace{1.6cm} +\left(  \int_{\ens}^{s} \left( \EE \Big[ \Big| K_1(s,r)\left(b(\enr,\Xb_{\eta_{n}(r)}^n)+b'(\enr,\Xb_{\eta_{n}(r)}^n) A_{r}^{1,n}\right)\Big|^p \Big] \right)^{\frac{1}{p}}~dr\right)^p\\
		&\hspace{1.6cm} +\EE \Big[ \Big( \int_{\ens}^s K_2(s,r)^2  ~ \left(\sigma'(\enr,\Xb_{\eta_{n}(r)}^n)  A_r^{n}\right)^2 ~dr \Big)^{\frac{p}{2}} \Big]
		+ \EE \big[ |  \epsilon^n_{s} |^p \big]
		 \bigg\} .
	\end{align*}
	Now one uses the boundedness of $b'$ and $\sigma'$, the bound $\EE| b(\enr, \Xb_{\eta_{n}(r)}^n) |^p\leq C \left(1 +\EE|X_{0}|^p\right)$ from Proposition \ref{prop:Milstein}, the bound on $\EE|A_{r}^n|^p$ from Proposition \ref{prop:Milstein}, and Minkowski's integral inequality on the second and fourth summand to get
	\begin{align*}
		&\EE \Big[ \Big| b'(\ens, \Xb^n_{\ens})(\Xb^n_{s}-\Xb^n_{\ens}-A_s^{n}) + \epsilon^n_{s} \Big|^p \Big]  \\
		\leq~ &
		C \EE \big[ |  \epsilon^n_{s} |^p \big]
		~+~
		C \big(1+\EE \big[|X_0|^p \big]\big) \Big(\int_0^{\ens} \left| K_1(s,r)-K_1(\ens,r) \right|~dr \Big)^p\\
		&+
		C \big(1+\EE \big[|X_0|^p \big]\big) \Big(\int_0^{\ens} \left(\EE \Big[ \Big| \left( K_2(s,r)-K_2(\ens,r) \right)  A_r^{n} \Big|^p \Big] \right)^{\frac{2}{p}} ~dr \Big)^{\frac{p}{2}}\\
		&+ C \big(1+\EE \big[|X_0|^p \big]\big) \Big(  \int_{\ens}^{s} | K_1(s,r)| ~dr\Big)^p
			+C  \big(1+\EE \big[|X_0|^p \big]\big) \Big( \int_{\ens}^s \Big( \EE\Big[ \big| K_2(s,r) ~  A_r^{n}\big|^p \Big] \Big)^{\frac{2}{p}} ~dr  \Big)^{\frac{p}{2}}.
	\end{align*}
	Observe that it follows from Proposition \ref{prop:Milstein} that $\EE \big[ |  \epsilon^n_{s} |^p \big] \leq C \dn^{2p(\alpha\wedge 1)}$ and that $\int_{\ens}^{s} | K_1(s,r)| ~dr\leq \dn^{\alpha_1 }$ by Condition \ref{eq:A2}.
	Then  the bound on $\EE \big[ |A_r^{n}|^p \big]$ from Proposition \ref{prop:Milstein},
	 together with conditions \ref{eq:A2} and \ref{eq:A3}, gives that
	\begin{align*}
		&
		\EE \Big[ \Big| b'(\ens, \Xb^n_{\ens})(\Xb^n_{s}-\Xb^n_{\ens}-A_s^{n}) + \epsilon^n_{s} \Big|^p  \Big] \\
		\leq~&
		C \Big(  \dn^{p (\alpha_1\wedge1)} + \dn^{2 p(\alpha_2 \wedge 1)} + \dn^{p \left(\alpha_1 \wedge 1\right)} + \dn^{2 p(\alpha\wedge 1)} + \dn^{2p(\alpha\wedge 1)}\Big)
		~\leq~
		C  \big(1+\EE \big[|X_0|^p \big]\big) \dn^{p(2\alpha' \wedge 1)} ,
	\end{align*}
	where we recall that $\alpha'$ was defined in Assumption \ref{assum:main2}.
	Plugging this bound in \eqref{eq:bound0I3}, it follows that
	\begin{align*}
		\EE \big[ |J_3|^p \big] &\leq  C \big(1+\EE \big[|X_0|^p \big]\big) \dn^{p(2\alpha' \wedge 1)} .
	\end{align*}
	For $J_4$, we denote by $\eta^+_n(r) := t^n_{k+1}$ for $r \in [t_k, t_{k+1})$,
	then by the boundedness of  $b'$, \ref{eq:A6}, Minkowski's integral inequality and the classical Fubini theorem,
	it follows that
	\begin{align*}
	    \EE \big[ |J_4|^p \big]  
	    &=
	    \EE \bigg[ \bigg| \int_0^t \int_r^{\eta^+_n(r)} K_1(t,s)K_2(s,r)b'(\eta_n(r), \Xb_{\eta_n(r)}^n)\sigma(\eta_{n}(r),\Xb^n_{\eta_{n}(r)})~ds ~dW_r \bigg|^p \bigg]\\
	    &\leq C\EE\bigg[ \int_0^t \Big| \int_r^{\eta^+_n(r)}K_1(t,s)K_2(s,r)b'( \eta_n(r), \Xb_{\eta_n(r)}^n)\sigma(\eta_{n}(r),\Xb^n_{\eta_{n}(r)})~ds \Big|^p dr  \bigg]\\
	    &\leq C \big(1+\EE \big[|X_0|^p \big]\big) \int_0^t \left(\int_r^{\eta^+_n(r)} \left| K_1(t,s)K_2(s,r) \right|~ds \right)^p ~dr\\
	    &\leq C \big(1+\EE \big[|X_0|^p \big]\big) ~ \dn^{p(2\alpha' \wedge 1)}.
	\end{align*}
	For $J_5$, we have by BDG's inequality and H\"{o}lder's inequality that
	$$\EE \big[ |J_5|^p \big]  \leq C \int_0^t\EE\big[  |X_s-\Xb_s^n|^p \big]~ds .$$
	For $J_6$, we have by BDG's inequality, Minkowski's integral inequality and Assumption \eqref{assum:main2} that
	\begin{align*}
		\EE \big[ |J_6|^p \big]
		~\leq~
		C\EE \Big[ \Big( \int_0^t |K_2(t,s)|^2(s-\ens)^{2\alpha' \wedge 1}(1+|\Xb_s^n|^2)~ds \Big)^{\frac{p}{2}} \Big]
		~\leq~
		C \big(1+\EE \big[|X_0|^p \big]\big) ~ \dn^{p(2\alpha' \wedge 1)} .
	\end{align*}
	The proof to bound $J_7$ is the same as for $J_{3}$: first, we have by the BDG inequality and Minkowski's integral inequality that
	\begin{align*}
		\EE \big[ |J_7|^p \big]
		\leq
		\left(\int_{0}^t  \left( \EE \Big[ \Big| K_{2}(t,s)\left( \sigma'(\ens, \Xb^n_{\ens})(\Xb^n_{s}-\Xb^n_{\ens}- A_s^{n} ) + \widetilde{\epsilon}^n_{s}\right) \Big|^p \Big] \right)^{\frac{2}{p}} ~ds\right)^{\frac{p}{2}},
	\end{align*}
	where $\widetilde{\epsilon}_{s}^n \leq C| \Xb^n_{s} - \Xb^n_{\ens} |^2 $ comes from the Taylor expansion of $\sigma$: $\sigma(\ens,\Xb_s^n)=\sigma(\ens,\Xb_{\ens}^n)+\sigma'(\ens,\Xb_{\ens}^n)(\Xb_s^n-\Xb_{\ens}^n)+\widetilde{\epsilon}_s^n$. Similarly to the computations made for $b$, it is clear that
	\begin{align*}
		\EE&\Big[ \Big| \sigma'(\ens, \Xb^n_{\ens})(\Xb^n_{s}-\Xb^n_{\ens}- A_s^{n}) + \widetilde{\epsilon}^n_{s} \Big|^p\Big]
		\leq
		C \big(1+\EE \big[|X_0|^p \big]\big) \dn^{p(2\alpha' \wedge 1)} .
	\end{align*}
	Hence
	\begin{align*}
		\EE \big[ |J_7|^p \big] \leq  C \big(1+\EE \big[|X_0|^p \big]\big) ~ \dn^{p(2\alpha' \wedge 1)}.
	\end{align*}
	In summary, one has
	$$
		\EE \big[ \big|X_t-\Xb_t^n \big|^p \big]
		\leq
		C \big(1+\EE \big[|X_0|^p \big]\big) \dn^{p(2\alpha' \wedge 1)}+C\int_0^t \EE \big[\big |X_s-\Xb_s^n \big|^p \big]~ds,
	$$
	and one can conclude with  Grönwall's Lemma.
	\qed

\subsection{Proof of Theorems \ref{thm:EulerScheme}.$(ii)$ and \ref{th:convMilstScheme}.$(ii)$}
\label{subsec:proofUnif}

\subsubsection{Garsia-Rodemich-Rumsey's estimates}\label{subsec:GRR}

	Let us first state the following consequences of Garsia-Rodemich-Rumsey's lemma \cite{GRR}.

	\begin{lemma}\label{lem:GRR}
		Let $\{Y_t, t\in[a,b]\}$ be an $\R^d$-valued continuous stochastic process on $[a, b] \subset \R$,
		Then for all $\gamma>0$, $p\geq 1\vee \gamma$ and $q>0$ such that $pq>2$,
		\begin{align*}
			\EE \bigg[ \sup_{t\in[a,b]} |Y_t-Y_a|^\gamma \bigg]
			&\leq
			\left(C \frac{pq}{pq-2} (b-a)^{q-\frac{2}{p}}\right)^\gamma\ \EE\left[ \left( \int_a^b \int_a^b \frac{|Y_s-Y_t|^p}{|t-s|^{pq}}\ d s\ d t\right)^{\frac{\gamma}{p}} \right] \\
			&\leq
			\left(C \frac{pq}{pq-2} (b-a)^{q-\frac{2}{p}}\right)^\gamma\ \left(\int_a^b \int_a^b \frac{\EE \big[ |Y_s-Y_t|^p \big]}{|t-s|^{pq}}\ d s\ d t\right)^{\frac{\gamma}{p}}.
		\end{align*}
	\end{lemma}
	\begin{proof}
		With the notations of \cite[p.353-354]{Nualart}, we apply the Garsia-Rodemich-Rumsey lemma with $\Psi(x) = x^p$ and $p(x)=x^q$ to obtain the first inequality.
		Then the second inequality follows by the H\"older's inequality.
	\end{proof}

	As a consequence of Lemma \ref{lem:GRR}, we easily deduce the following corollary.
	\begin{corollary}\label{cor:GRR}
		Let $(Y^n)_{n\ge 1}$ be a sequence of continuous processes on $[0,T]$.
		Assume that there exist constants $\gamma>0$, $p\geq 1\vee \gamma$, $\eta>1$, $\rho>0$, $C_{0} > 0$
		and a sequence $(\delta_{n})_{n\ge 1}$ of positive real numbers such that
		\begin{equation*}
			\EE \Big[ \big|Y^n_{s} - Y^n_{t} \big|^p \Big]
			~\leq~
			C_{0}\,  |s-t|^\eta\, \delta_{n}^{\rho}, \quad \forall s,t\in[0,T], ~\forall n \ge 1.
		\end{equation*}
		Then there exists a constant $C_{p,\gamma, \eta,T}>0$, depending only on $p$, $\gamma$, $\eta$ and $T$,
		such that $\forall n \ge 1$,
		\begin{equation*}
			\EE \Big[ \sup_{t\in[0,T]} \big| Y^n_t-Y^n_0 \big|^\gamma \Big]
			\leq
			C_{p,\gamma,T} ~ C_{0}^{\frac{\gamma}{p}} ~ \delta_{n}^{\frac{\rho \gamma}{p}}.
		\end{equation*}
	\end{corollary}

\subsubsection{Proof of Theorems \ref{thm:EulerScheme}.$(ii)$ and \ref{th:convMilstScheme}.$(ii)$}

	$(i)$ Let us first consider Theorems \ref{thm:EulerScheme}.$(ii)$,
	for which we will apply Corollary \ref{cor:GRR} to  $Y^{n}:= X-X^{n}$, with the solution  $X^{n}$ to the Euler scheme \eqref{eq:EulerScheme}.
	
	\vspace{0.5em}

	Let  $\theta\in (0,1)$, one has
	\begin{align*}
		\Big(\EE \big[ \big| Y^{n}_{s} - Y^{n}_{t} \big|^p \big] \Big)^{\frac{1}{p}}
		&\leq
		\left(
			\Big( \EE \Big[ \big| X_{s} - X_{t} \big|^p \Big] \Big)^{\frac{1}{p}}
			+
			\Big( \EE \Big[ \big| X^{n}_{s} - X^{n}_{t} \big|^p \Big] \Big)^{\frac{1}{p}}
		\right)^\theta\\
		&\quad \quad
		\x
		\left(
			\Big( \EE \Big[ \big| X_{t} - X^{n}_{t} \big|^p \Big] \Big)^{\frac{1}{p}}
			+
			\Big(\EE \Big[ \big| X_{s} - X^{n}_{s} \big |^p \Big] \Big)^{\frac{1}{p}}
		\right)^{1-\theta}.
	\end{align*}
	The two first terms on the r.h.s. can be controlled using Proposition \ref{prop:Euler},
	and the two last terms would be controlled using Theorem \ref{thm:EulerScheme}.$(i)$,
	and it follows that, for some constant $C$ depending on $p$ but independent of $n$ 
	and $\E[ |X_0|^p]$,
	\begin{equation} \label{eq:sup_norm_interm}
		\left(\EE \Big[  \big| Y^{n}_{s} - Y^{n}_{t} \big |^p \Big]\right)^{\frac{1}{p}}
		~\leq~
		C \, \big(1+ \E\big[ |X_{0}|^p\big] \big)^{\frac{1}{p}}\, |s-t  |^{(\alpha \wedge 1)\theta} \times \dn^{( \alpha \wedge 1) (1-\theta)} .
	\end{equation}
	For any $\gamma\in(0,p]$ and $\varepsilon\in (\frac{1}{p}, \alpha\wedge 1)$,
	one can set $\theta := \varepsilon (\alpha\wedge 1)^{-1}$
	and $\eta := p\varepsilon >1$,
	and let  $\rho \equiv p ( \alpha \wedge 1) (1-\theta)$.
	Then the estimation in \eqref{eq:sup_norm_interm} satisfies the conditions in Corollary \ref{cor:GRR},
	and it follows that, for some constant $C_{\gamma, \varepsilon} > 0$,
	\begin{align*}
		\bigg(\EE \Big[ \sup_{t\in[0,T]} \big|X_t-X^{n}_{t} \big|^\gamma \Big] \bigg)^{\frac{1}{\gamma}}
		&=
		\bigg(\EE \Big[ \sup_{t\in[0,T]} \big|Y^n_t - Y^n_0 \big|^\gamma \Big] \bigg)^{\frac{1}{\gamma}} \\
		&\leq
		C_{\gamma, \varepsilon}  \left(1+ \E\big[ |X_{0}|^p\big] \right)^{\frac{1}{p}} ~\dn^{( \alpha \wedge 1) (1-\theta)}
		=
		C_{\gamma, \varepsilon} \left(1+\E\big[ |X_{0}|^p\big]\right)^{\frac{1}{p}} ~\dn^{( \alpha \wedge 1) -\varepsilon},
	\end{align*}
	which proves Theorems \ref{thm:EulerScheme}.$(ii)$.
	
	\vspace{0.5em}
	
	\noindent $(ii)$ The proof for Theorem  \ref{th:convMilstScheme}.$(ii)$ is similar.
	It is enough to apply Corollary \ref{cor:GRR} on $\overline Y^n := X - \Xb^n$.
	In place of the estimations in Proposition \ref{prop:Euler} and Theorem \ref{thm:EulerScheme}$.(i)$,
	one can use those in Proposition \ref{prop:Milstein} and Theorem \ref{th:convMilstScheme}.$(i)$
	to obtain that
	\begin{equation*}
		\left(\EE \Big[  \big| \overline Y^{n}_{s} - \overline Y^{n}_{t} \big |^p \Big]\right)^{\frac{1}{p}}
		~\leq~
		C \, \left(1+ \E\big[ |X_{0}|^p\big]p\right)^{\frac{1}{p}}\,  |s-t  |^{(\alpha \wedge 1)\theta} \times \dn^{( 2 \alpha' \wedge 1) (1-\theta)},
	\end{equation*}
	for some constant $C$ independent of $n$.
	The rest of the proof is then almost the same as in item $(i)$.
	\qed


\end{document}